\documentclass{article}

\usepackage[12pt]{extsizes}
\usepackage{cmap}
\usepackage[utf8]{inputenc}
\usepackage[T2A]{fontenc}
\usepackage[english]{babel}
\usepackage{amsmath}
\usepackage{amsthm}
\usepackage{amssymb}
\usepackage{amsfonts}
\usepackage{euscript}
\usepackage{enumerate}
\usepackage{hyperref}
\usepackage[affil-it]{authblk}

\usepackage[all]{xy}
\usepackage{epigraph}

\newcommand{\RN}{\mathbb{R}} 
\newcommand{\CN}{\mathbb{C}} 
\renewcommand{\C}{\mathbb{C}}
\newcommand{\ZN}{\mathbb{Z}} 

\newcommand{\eps}{\ensuremath\varepsilon}
\newcommand{\intl}{\int\limits}
\newcommand{\suml}{\sum\limits}

\newcommand{\plusl}{\bigoplus\limits}
\newcommand{\ovl}{\overline}

\newcommand{\gr}{\mathop{\mathrm{gr}}\nolimits}

\newcommand{\SL}{\mathop{\mathrm{SL}}\nolimits}

\renewcommand{\Im}{\mathop{\mathrm{Im}}\nolimits}

\newcommand{\ad}{\ensuremath\operatorname{ad}}
\newcommand{\id}{\ensuremath\operatorname{id}}

\newcommand{\mf}{\mathfrak}

\newcommand{\mc}[1]{\mathcal{#1}}

\renewcommand{\Re}{\operatorname{Re}}

\newcommand{\abs}[1]{\lvert #1\rvert}
\newcommand{\bP}{{\bold P}}

\renewcommand{\le}{\leqslant}
\newcommand{\Mat}{\operatorname{Mat}}
\begin{document}
\newtheorem{thr}{Theorem}[section]
\newtheorem*{thr*}{Theorem}
\newtheorem{lem}[thr]{Lemma}
\newtheorem*{lem*}{Lemma}
\newtheorem{cor}[thr]{Corollary}
\newtheorem{prop}[thr]{Proposition}
\newtheorem*{prop*}{Proposition}
\newtheorem{stat}[thr]{Statement}
\newtheorem*{stat*}{Statement}
\newtheorem{example}[thr]{Example}

\theoremstyle{definition}
\newtheorem{defn}[thr]{Definition}
\newtheorem*{defn*}{Definition}
\theoremstyle{remark}
\newtheorem{rem}[thr]{Remark}
\newtheorem*{rem*}{Remark}
\numberwithin{equation}{section}
\author{Daniil Klyuev}
\title{Unitarizability of Harish-Chandra bimodules over generalized Weyl and $q$-Weyl algebras}
\maketitle
\begin{abstract}
Let $\mc{A}$ be a quantized ($K$-theoretic) BFN Coulomb branch with $G=\C^*$ and any $N$, that is, $\mc{A}$ is generalized Weyl or $q$-Weyl algebra. Let $M$ be an $\mc{A}-\ovl{\mc{A}}$-bimodule. Choosing an antilinear automorphism $\rho$ of $\mc{A}$ we can define the notion of an invariant Hermitian form on $M$: $(au,v)=(u,v\rho(a))$ for all $u,v\in M$ and $a\in\mc{A}$. We obtain a classification of invariant positive definite forms on $M$ in the case when $M$ is Harish-Chandra in the sense of Losev and quantization parameter is generic.
\end{abstract}
\section{Introduction}


Let $\mc{A}$ be an algebra over $\CN$, $M$ be an $\mc{A}-\ovl{\mc{A}}$-bimodule. Choose an automorphism of $\mc{A}$. It gives an antilinear isomorphism of algebras $\rho\colon \mc{A}\to\ovl{\mc{A}}$. A Hermitian form $(\cdot,\cdot)$ on $M$ is said to be {\it $\rho$-invariant} if $(am,n)=(m,n\rho(a))$ for all $m,n\in M$, $a\in \mc{A}$. 

The notion of {\it Harish-Chandra} bimodule~\cite{LosevHarishChandra} is a generalization of a classical notion of $(\mf{g},K)$-module. For an algebra $\mc{A}$ and a Harish-Chandra bimodule $M$ one can ask the following question: what are the invariant positive definite forms on $M$? For example, let $\mf{g}$ be a complex simple Lie algebra, $G$ be the corresponding simply-connected group. If we take $\mc{A}$ to be  a central reduction of $U(\mf{g})$, the classification of irreducible Harish-Chandra bimodules $M$ that have an invariant positive definite form is equivalent to the classification of irreducible unitary representations of $G$ with this central character.

We will consider the following case. Let $A$ be an algebra of function of a {\it Kleinian singularity of type $A$}, meaning $A=\CN[x,y]^{\ZN/n}=\CN[x^n,y^n,xy]$. We take $\mc{A}$ to be a filtered deformation or a $q$-deformation of $A$. 

Filtered deformations $\mc{A}$ of $\CN[x,y]^{C_n}$ are parametrized by monic polynomials $P$ of degree $n$: deformation $\mc{A}_P$ is generated by $u,v,z$ with relations \[[z,u]=-u,\quad [z,v]=v, \quad uv=P(z+\tfrac12), \quad vu=P(z-\tfrac12).\] These algebras are also called generalized Weyl algebras: when $n=1$, $P(x)=x$, $\mc{A}$ is the usual Weyl algebra with generators $u,v$ and relation $[u,v]=1$.

Similarly, $q$-deformations $\mc{A}$ of $\CN[x,y]^{C_n}$ correspond to Laurent polynomials $P$: deformation $\mc{A}_P$ is generated by $u,v,Z$ with relations 
\[ZuZ^{-1}=q^2u,\quad ZvZ^{-1}=q^{-2}v,\quad uv=P(q^{-1}Z),\quad vu=P(qZ).\] These algebras are also called generalized $q$-Weyl algebras: when $P(x)=1$, $\mc{A}$ is the $q$-Weyl algebra with generators $u^{\pm 1}, Z^{\pm 1}$ and relation $Zu=q^2uZ$.

Papers~\cite{EKRS},~\cite{K} classify positive definite invariant forms in the case when $\mc{A}$ is isomorphic to $\ovl{\mc{A}}$ and $M$ is the {\it regular} bimodule, meaning $M=\mc{A}$ with the natural action.
 
 In this paper we will classify positive definite invariant forms on certain bimodules over ($q$-)deformations of $\CN[x,y]^{C_n}$ in the case of {\it generic} parameter. In the case of a filtered deformation this means that all roots of $P(x)$ are distinct and no two roots can differ by an integer. In this case we have a construction of $\mc{A}$ as a certain subalgebra inside $\CN[x,x^{-1},\partial_x]$ and $M$ as a certain subspace of $\CN[x,x^{-1},\partial_x]$. We show that the modules $M$ we constructed are Harish-Chandra and, moreover, we use Losev's description of category of Harish-Chandra bimodules over $A$ to show that all simple Harish-Chandra bimodules can be constructed in this way.
 
We also do the following. Paper~\cite{ES} defined the notion of a short star-product. This notion can be extended to a not necessarily commutative algebras $A$ such that $A_0$ is a semisimple algebra. We prove some properties of short star-products in this case. In particular, each positive trace on $M$ gives a short star-product on the commutative algebra $\begin{pmatrix}
\gr\mc{A} & \gr M\\
\gr\ovl{M} & \gr\ovl{\mc{A}}
\end{pmatrix}$. Here both $\gr\mc{A}$ and $\gr\ovl{\mc{A}}$ are isomorphic to $\CN[x,y]^{C_n}$ and $\gr M$ is isomorphic to a $\CN[x,y]^{C_n}$-submodule of $\CN[x,x^{-1},y]$. Similarly to~\cite{EKRS} and~\cite{K} we expect that these short star-products  are useful in 3-dimensional superconformal field theories~\cite{BPR} and in the study of the Coulomb branch of 4-dimensional superconformal field theories~\cite{DG}.

The organization of paper is as follows. In Section~\ref{SecStarProducts} we introduce the notion of a short star-product in more general case. We have an analogue of Theorem 3.1 in~\cite{ES} in this case: there is a correspondence between short star-products on $A$ and {\it twisted traces} on filtered deformations of $A$. For an automorphism $g$ of $\mc{A}$ by $g$-twisted trace we mean a linear map $T\colon A\to\CN$ such that $T(ab)=T(bg(a))$ for all $a,b\in \mc{A}$. We also prove that under certain condition twisted traces on a matrix algebra $\begin{pmatrix}
\mc{A} & M\\
\ovl{M} & \ovl{\mc{A}}
\end{pmatrix}$ correspond to invariant sesquilinear forms on $M$. Combining these two results we obtain short star-products from positive definite invariant forms.

In Section~\ref{SecUsualUnitarizability} we classify positive definite invariant forms in the case when $\mc{A}$ is a filtered deformation of $\CN[x,y]^{C_n}$. In order to do this first we note that invariant forms on $M$ are in one-to-one correspondence with twisted traces on $L=M\otimes_{\ovl{\mc{A}}}\ovl{M}$, where twisted traces are defined similarly to the above. Then we prove that $L$ is isomorphic to $\mc{A}$ as an $\mc{A}$-bimodule. We know the classification of twisted traces on $\mc{A}$ from~\cite{EKRS}. For generic parameter each trace can be expressed as a certain contour integral. This and the construction of $M$ as a subspace of $\CN[x,x^{-1},\partial_x]$ allow us to compute the set of traces that give positive definite form on $M$.

The answer in~\cite{EKRS} depends on the number of roots $\alpha$ with $\abs{\Re\alpha}<\tfrac12$ and the number of roots with $\abs{\Re\alpha}=\tfrac12$. In our situation the module $M$ exists only when for each index $i$ from $1$ to $n$ there exists an index $j$, possibly equal to $i$ such that $\alpha_i+\ovl{\alpha_j}$ is an integer. We say that $i$ is {\it good} if the real number $\alpha_i-\alpha_j$ satisfies $\abs{\alpha_i-\alpha_j}<1$. Note that in the case $M=\mc{A}$ we have $\alpha_i+\ovl{\alpha_j}=0$, so that $\abs{\alpha_i-\alpha_j}=2\abs{\Re\alpha_i}$. Hence the notion of a good root is an analogue of a root $\alpha$ with $\abs{\Re\alpha}<\tfrac 12$. Since our parameter is generic we don't have an analogue of a root $\alpha$ with $\abs{\Re\alpha}=\tfrac12$. In the end we get the same answer as in~\cite{EKRS}: invariant positive definite forms are a convex cone of dimension equal to the number of good roots minus a constant from $0$ to $4$ that depends on $n$ and $\rho$.

In Section~\ref{SecQUnitarizability} we classify positive definite invariant forms in the case when $\mc{A}$ is a $q$-deformation of $\CN[x,y]^{C_n}$. The reasoning here is similar to the reasoning in Section~\ref{SecUsualUnitarizability}. The answer is always a convex cone of dimension equal to the number of good roots. 

In the case of filtered deformation and $n=2$ the algebra $\mc{A}$ is isomorphic to a central reduction of $U(\mf{sl}_2)$. In the case of $q$-deformation and $P(x)=ax+b+cx^{-1}$ the algebra $\mc{A}$ is isomorphic to a central reduction of $U_q(\mf{sl}_2)$. We check that our results in these cases give the same answer as the classical results on irreducible unitary representations of $\SL(2,\CN)$ and $\SL_q(2)$~\cite{Pusz}.

\subsection{Acknowledgements}
I am grateful to Pavel Etingof for suggesting the problem and helpful comments on the previous versions of this paper. The first version of this article was written while I was a graduate student at MIT.
\section{Generalized star-products}
\label{SecStarProducts}
Let $A=\bigoplus\limits_{k\geq 0} A_k$ be a positively graded algebra over $\CN$, where all graded pieces $A_k$ are finite-dimensional and $A_0$ is a semisimple algebra.

We assume that filtered algebras have a $\ZN_{\geq 0}$ filtration.
\begin{defn}
A map $*\colon A\times A\to A$ is called a star-product if $(A,*)$ is an associative algebra and  for all homogeneous $a,b\in A$ we have \[a*b=ab+\suml_{k>0}C_k(a,b),\] where $\deg C_k(a,b)=\deg a+\deg b-2k$.
\end{defn}
Similarly to~\cite{ES} star-products are in one-to-one correspondence with $\ZN/2\ZN$-equivariant filtered quantizations equipped with a quantization map.

\begin{defn}
A filtered deformation of $A$ is a pair of a filtered algebra $\mc{A}$ and an isomorphism between $\gr\mc{A}$ and $A$.
\end{defn}
In particular, $\mc{A}_{\le 0}$ is identified with $A_0$. Usually we will not mention the isomorphism and write ``$\mc{A}$ is a filtered deformation of $A$''.

\begin{defn}
A $\ZN/2\ZN$-equivariant quantization of $A$ is a pair $(\mc{A},s)$, where $\mc{A}$ is a filtered quantization of $A$ and $s$ is an involution of $\mc{A}$ such that $\gr s=(-1)^d$, meaning it acts on $A_k$ as $(-1)^k$ for all $k$.
\end{defn}
\begin{defn}
A quantization map is a linear map $\phi\colon A\to \mc{A}$ such that $\phi(A_k)\subset \mc{A}_{\le k}$ and $\gr\phi$ is the identity. We say that $\phi$ is $\ZN/2\ZN$-equivariant if $s(\phi(a))=(-1)^d\phi(a)$ for all homogeneous $a\in A$.
\end{defn}

We have the following
\begin{prop}
Star-products on $A$ are in one-to-one correspondence with pairs $(\mc{A},\phi)$, where $\mc{A}$ is a $\ZN/2\ZN$-equivariant filtered deformation of $A$ and $\phi$ is a $\ZN/2\ZN$-equivariant quantization map.
\end{prop}
\begin{proof}
If $*$ is a star-product on $A$ then we define $\mc{A}$ to be $(A,*)$, $s=(-1)^d$ and $\phi$ to be the identity map. On the other hand if we have $(\mc{A},\phi)$, we define $a*b:=\phi^{-1}(\phi(a)\phi(b))$. In both cases all of the required conditions are satisfied.
\end{proof}
\begin{defn}
We say that a star-product $*$ is {\it short} if for all homogeneous $a,b\in A$ and $k>\min(\deg a,\deg b)$ we have $C_k(a,b)=0$.
\end{defn}
We note that if $*$ is a short star-product on $A$ then for any $k\geq 0$, $a\in A_0, b\in A_k$ we have $a*b=ab$ and $b*a=ba$. Hence the structure of an $A_0$-bimodule on $(A,*)$ coincides with the structure of an $A_0$-bimodule on $A$.
\begin{defn}
Let $*$ be a short star-product on $A$. It is called nondegenerate if for all $k\geq 0$ the left kernel of the bilinear map $C_k\colon A_k\times A_k\to A_0$ is zero.
\end{defn}

It follows from the next proposition and its proof that if the left kernel of $C_k$ if zero then the right kernel of $C_k$ is zero and vice versa. 

\begin{prop}
Let $T_0\colon A_0\to \CN$ be a nondegenerate trace: $T_0(ab)=T_0(ba)$ for all $a,b$ and $(a,b)\mapsto T_0(ab)$ is a symmetric nondegenerate bilinear pairing on $A_0\times A_0$. Then for all $k\geq 0$ the map $T_0\circ C_k$  gives a nondegenerate bilinear form on $A_k$.
\end{prop}
\begin{proof}
Let $V_1\ldots,V_l$ be all simple representations of $A_0$, so that $A_0=\plusl_{i=1}^l R_i$, where $R_i=V_i\otimes V_i^*$ are matrix algebras. Fix $k\geq 0$. Let $A_k=\bigoplus\limits_{i=1}^l V_i^{d_i}$ as a left $A_0$-module and $\plusl_{i=1}^l (V_i^*)^{e_i}$ as a right $A_0$-module. 

Since $T_0$ is a trace, its restriction to $V_i\otimes V_i^*$ is given by $T_0(u\otimes v)=\alpha_i v(u)$, where $\alpha_i$ is some complex number. Since $T_0$ is nondegenerate, $\alpha_i$ is nonzero.

Let $a\in A_0$, $b,c\in A_k$. Then $(ab)*c=(a*b)*c=a*(b*c)=a(b*c)$. Comparing $C_k$ we get $C_k(ab,c)=a C_k(b,c)$. Similarly $C_k(b,ca)=C_k(b,c)a$. Define a structure of an $A_0$-bimodule on $A_k\otimes_{\CN} A_k$ as follows: $a(b\otimes c):=ab\otimes c$, $(b\otimes c)a:=b\otimes ca$. If we regard $C_k$ as a linear map from $A_k\otimes_{\CN} A_k$ to $A_0$, it will be a map of $C_k$-bimodules.

We have $A_k\otimes A_k=\bigoplus_{i,j=1}^l (V_i\otimes V_j^*)^{d_ie_j}$. Using Schur Lemma we see that $C_k$ restricted to $V_i\otimes V_j^*$ is zero when $i\neq j$. When $i=j$, the map $C_k$ is defined by a matrix $M_i$ of size $d_i\times e_i$: $C_k(u\otimes v)=(M_i)_{ab}(u\otimes v)$ for $u$ in the $a$-th copy of $V_i$ and $v$ the in $b$-th copy of $V_i^*$. 

Since $C_k$ has no left kernel, it follows that $M_i$ has rank at least $d_i$. In particular, $d_i\geq e_i$. Computing the dimension of $A_k$ in two different ways, we get $\suml_{i=1}^l d_i\dim V_i=\suml_{i=1}^l e_i\dim V_i$. It follows that $d_i=e_i$. We get that $M_i$ is a nondegenerate square matrix.

Hence $T_0\circ C_k$ is a nondegenerate bilinear form on $A_k$, as we wanted.
\end{proof}
From now on we fix a nondegenerate trace $T_0$ on $A_0$.

Let $\langle a,b\rangle=T(a*b)=T_0(CT(a*b))$, where $CT$ is the constant term map. The shortness of $*$ implies that $A_k$ is orthogonal to $A_m$ with respect to $\langle\cdot,\cdot\rangle$ when $k\neq m$. Hence if $*$ is nondegenerate then $\langle \cdot,\cdot\rangle$ is a nondegenerate bilinear form.

Similarly to Proposition~3.3 in~\cite{ES} we want a bijection between nondegenerate short star-products and nondegenerate twisted traces.

\begin{defn}
Let $\mc{A}$ be a filtered quantization of $A$ and $g$ be an automorphism of $\mc{A}$ such that $g$ restricted to $\mc{A}_{\le 0}$ is identity. We say that a map $T\colon\mc{A}\to \CN$ is a {\it $g$-twisted trace} if $T(ab)=T(bg(a))$ for all $a,b\in\mc{A}$. We say that $T$ {\it extends} $T_0$ if $T|_{\mc{A}_0}=T_0$.
\end{defn}
\begin{defn}
We say that $T$ is {\it nondegenerate} if $(a,b)=T(ab)$ is nondegenerate on each $\mc{A}_{\le i}$.
\end{defn}

In~\cite{ES} the trace $T_0$ is defined by one number $T_0(1)=T(1)$, so it is clear that different choices $T_0$ give essentially the same set of traces. In our case after changing $T_0$ we should multiply both $T|_{A_i}$ and  $g|_{A_{ij}}$ by  nonzero numbers.

We have the following analog of Proposition~3.3 in~\cite{ES}:
\begin{prop}
\label{PropStarProductsAndTwistedTraces}
There is a one-to-one correspondence between nondegenerate star-products on $A$ and triples $(\mc{A},g,T)$, where $\mc{A}$ is a filtered deformation of $A$, $g$ is an automorphism of $\mc{A}$ equal to identity on $\mc{A}_{\le 0}$ and $T$ is an $s$-invariant $g$-twisted trace that extends $T_0$.
\end{prop}
\begin{proof}
Let $*$ be a nondegenerate star-product on $A$, $\mc{A}=(A,*)$. Define $T=T_0\circ CT$, where $CT$ is the constant term map. This map commutes with $s=(-1)^d$. For any $k$, since $\langle\cdot,\cdot\rangle$ is nondegenerate on $A_k$, there exists a linear automorphism $g_k$ such that $\langle a,b\rangle=\langle g_k(b),a\rangle$ for all $a,b\in A_k$. This means $T(ab)=T(bg_k(a))$ for all $a,b\in A_k$. We define $g$ so that $g|_{A_k}=g_k$.  In this way $T$ becomes a $g$-twisted trace. It remains to prove that $g$ is an automorphism of $\mc{A}$.

For all $a,b,c\in \mc{A}$ we have $T(abc)=T(bcg(a))=T(cg(a)g(b))$ on one hand and $T(abc)=T(cg(ab))$ on the other hand. Since $T$ has zero kernel, we get $g(a)g(b)=g(ab)$. Hence $g$ is an automorphism of $\mc{A}$.

Since $T_0$ is a trace on $A_0$, the restriction of $g_0$ to $A_0$ is the identity.

Since $\langle\cdot,\cdot\rangle$ is nondegenerate and $A_i$ is orthogonal to $A_j$ when $i\neq j$, $\langle\cdot,\cdot\rangle$ is nondegenerate on $A_0\oplus\cdots\oplus A_k=\mc{A}_{\le k}$. This precisely means that $T$ is nondegenerate.

On the other hand, suppose that we have $(\mc{A},g,T)$ as above.

 Define $(a,b)=T(ab)$. For all $a,b\in \mc{A}$ we have $(a,b)=(b,g(a))$. Using this for $a\in \mc{A}_{\le k}$ we see that the left and the right orthogonal complement of $\mc{A}_{\le k}$ coincide. For $k\geq 0$ define $\mc{A}_k=\mc{A}_{\le k-1}^{\perp}\cap\mc{A}_{\le k}$. 

Since $T$ is nondegenerate, $\mc{A}_k$ has trivial intersection with $\mc{A}_{\le k-1}$. Counting dimensions we have $\dim\mc{A}_k=\dim\mc{A}_{\le k}-\dim\mc{A}_{\le k-1}=\dim A_k$. It follows that $\mc{A}_k$ is in natural bijection with $\mc{A}_{\le k}/\mc{A}_{\le k-1}=A_k$.

Define $\phi\colon A\to \mc{A}$ using these bijections. Define $*$ by $a*b=\phi^{-1}(\phi(a)\phi(b))$. Since $(A,*)$ is isomorphic to $\mc{A}$, $*$ is a star-product.

Let us prove that $\phi$ is $\ZN/2\ZN$-invariant. Let $a$ be an element of $\mc{A}_k$. Note that $(sa,b)=T(sa\cdot b)=T(a\cdot sb)=(a,sb)$ since $T$ is $s$-invariant. Since $s$ is filtration-preserving, we get $sa\in \mc{A}_{\le k}\cap \mc{A}_{\le k-1}^{\perp}=\mc{A}_k$. On the other hand, since $\gr s=(-1)^d$, $a+(-1)^{k+1}sa$ belongs to $\mc{A}_{\le k-1}$. We deduce that $sa=(-1)^k a$ when $a\in A_k$. It follows that $\phi$ is $\ZN/2\ZN$-invariant.

It remains to check that $*$ is short and nondegenerate. 

We start with shortness. Let $a\in A_k,b\in A_l$. Suppose that $k>l$. 

We note that for $r>0$ we have $\mc{A}_{\le r}^{\perp}=\plusl_{t>r} \mc{A}_t$. So in order to check that $a*b$ belongs to $\plusl_{s\geq k-l} A_s$ it is enough to check that $\phi(a*b)=\phi(a)\phi(b)$  belongs to $\mc{A}_{\le k-l-1}^{\perp}$. 

Let $c\in \mc{A}_{\le k-l-1}$. We want to check that $T(\phi(a)\phi(b)c)=0$. We note that $\phi(b)c\in \mc{A}_{\le k-1}$. Using the definition of $\mc{A}_k$, we deduce that $T(\phi(a)\phi(b)c)=0$.

It remains to prove that $*$ is nondegenerate. Since $T$ is nondegenerate, $(\cdot,\cdot)$ is nondegenerate on each $\mc{A}_{\le k}$. Using the definition of $\mc{A}_{ k}$, we see that $(\cdot,\cdot)$ is nondegenerate on each $A_k$. It follows that $*$ is nondegenerate.

Finally, it is straightforward to check that the maps $*\mapsto (\mc{A},g,T)$ and $(\mc{A},g,T)\mapsto *$ are inverse to each other.
\end{proof}
\subsection{Example: $2\times 2$ matrices}
\label{SubSecExampleTwoByTwo}
Suppose that $A_0$ is either the algebra of $2\times 2$ matrices or diagonal $2\times 2$ matrices. Let $e_1,e_2$ be the diagonal matrix units. In this case $A=A_{11}\oplus A_{12}\oplus A_{21}\oplus A_{22}$, where $A_{ij}=e_i A e_j$. We see that $A_{11}$ and $A_{22}$ are subalgebras, $A_{12}$ is an $A_{11}-A_{22}$ bimodule, $A_{21}$ is an $A_{22}-A_{11}$ bimodule, and there are maps of bimodules $\phi\colon A_{12}\otimes_{A_{22}}A_{21}\to A_{11}$, $\psi\colon A_{21}\otimes_{A_{11}}A_{12}\to A_{22}$. In this case $A_{11},A_{22},A_{21},A_{12}$ form a Morita context.

Let $\mc{A}$ be a filtered deformation of $A$. Then $\mc{A}_{\le 0}=A_0$, so we can similarly define $\mc{A}_{11}$, $\mc{A}_{22}$, $\mc{A}_{12}$, $\mc{A}_{21}$. Since $\mc{A}$ is a filtered deformation, for any $e\in A_0$ and $a\in \mc{A}_{\le k}$ we have $e(a+\mc{A}_{\le k-1})=ea+\mc{A}_{\le k-1}$. It follows that the isomorphism $\gr\mc{A}\cong A$ sends $\gr\mc{A}_{ij}$ to $A_{ij}$.

Let $g$ be an automorphism of $\mc{A}$ that restricts to the identity on $A_0$. From $g(e_ia)=e_ig(a)$ and $g(ae_i)=g(a)e_i$ we deduce that $g$ preserves $\mc{A}_{ij}$ for all pairs $(i,j)$.

Let $T_0$ be the standard matrix trace. Any twisted trace $T$ on $\mc{A}$ satisfies $T(e_ia)=T(ae_i)$. It follows that $T$ is zero on $\mc{A}_{12}$ and $\mc{A}_{21}$. Hence for a fixed $g$ a $g$-twisted trace $T$ is supported on $\mc{A}_{11}\cup \mc{A}_{22}$ and satisfies $T(ab)=T(gb(a))$ for all pairs $i,j\in\{1,2\}$, $a\in A_{ij}, b\in A_{ji}$.

The condition $T|_{A_0}=T_0$ is equivalent to $T(e_1)=T(e_2)=1$.

Using $T(ab)=T(ba)$ for $a\in A_0$, $b\in\mc{A}$ we can prove that $\mc{A}_k=\mc{A}_{\le k}\cap \mc{A}_{\le k-1}^{\perp}$ is an $A_0$-submodule of $\mc{A}$. It follows that the quantization map $\phi$ is an isomorphism of $A_0$-modules. Hence $A_{ij}*A_{kl}$ is a subset of $A_{il}$ when $j=k$ and zero otherwise. Note that star-products in general do not have to satisfy this property if we take $\phi$ that does not respect the $A_0$-module structure.

\subsection{Conjugations, Hermitian star-products, positive traces}
\label{SubSecConjugations}
The notions of conjugation-invariant and Hermitian star-products in~\cite{ES} can be defined in our situation. Similarly to~\cite{ES} there is a connection between Hermitian star-products and Hermitian forms; we will work it out in the example below. This gives a connection between our unitarizability results below and star-products: any unitarizable bimodule $M_{c,c'}$ will give a star-product on the algebra $\begin{pmatrix}
A & I\\
I & A
\end{pmatrix}$, where $A=\CN[x,y]^{C_n}$ and $I$ is a $\CN[x,y]^{C_n}$-submodule of $\CN[x,x^{-1},y]$.


\begin{defn}
\begin{enumerate}
\item
A conjugation on $A$ is an antilinear graded automorphism $\rho\colon A\to A$.
\item
Let $\mc{A}$ be a $\ZN/2$-invariant deformation of $A$. A conjugation of $\mc{A}$ is a filtered antilinear automorphism $\rho$ that commutes with $s$.
\end{enumerate}
\end{defn}
Recall that $A_0$ is a semisimple algebra. Hence $A_0$ is a direct product of matrix algebras and any automorphism of $A_0$ is a composition of inner automorphism and a permutation of matrix algebras of the same size. The antilinear automorphism $\rho|_{A_0}$ is a composition of the standard complex conjugation and an automorphism of $A_0$. Choose $T_0$ satisfying the following two conditions: $T_0$ is invariant under all automorphisms of $A_0$ and $T_0(\ovl{a})=\ovl{T_0(a)}$, where $\ovl{a}$ is the standard complex conjugation on matrices. For example, $T_0$ that is equal to the standard matrix trace on each matrix algebra satisfies these conditions. It follows that $T_0(\rho(a))=\ovl{T_0(a)}$ for any conjugation $\rho$ of $A$.

From now on we fix such $T_0$.

\begin{defn}
\begin{enumerate}
\item
A star-product $*$ on $A$ is conjugation-invariant if\linebreak $\rho(a*b)=\rho(a)*\rho(b)$.
\item
A star-product $*$ on $A$ is Hermitian if $T_0(CT(a*b))=T_0(CT(b*\rho^2(a)))$ for all $a,b\in A$.
\end{enumerate}
\end{defn}
We have a lemma similar to Lemma~3.20 in~\cite{ES}.
\begin{lem}
Let $*$ be a conjugation-invariant nondegenerate star-product on $A$. Then $*$ is Hermitian if and only if the form $(a,b)\mapsto T_0\circ CT(a*\rho(b))$ is Hermitian.
\end{lem}
\begin{proof}
The form $(\cdot,\cdot)$ is Hermitian when $(a,b)=\ovl{(b,a)}$ for all $a,b\in A$. We know that $T_0(\rho(a))=\ovl{(T_0(a))}$ for $a\in A_0$. Since $\rho$ is graded, we have $CT(\rho(a))=\rho(CT(a))$ for any $a\in A$. Using this we get

\[\ovl{(b,a)}=\ovl{T_0\circ CT(b*\rho(a))}=T_0\big(\rho(CT(b*\rho(a)))\big)=T_0(CT(\rho(b)*\rho^2(a))).\]

 So $(\cdot,\cdot)$ is Hermitian if and only if

\[T_0(CT(a*\rho(b)))=T_0(CT(\rho(b)*\rho^2(a)))\]

for all $a,b\in A$.

Changing $b$ to $\rho^{-1}(b)$ we see that $(\cdot,\cdot)$ is Hermitian if and only if 

\[T_0(CT(a*b))=T_0(CT(b*\rho^2(a)))\] for all $a,b\in A$. This means that $*$ is Hermitian.
\end{proof}

We have a lemma similar to Lemma~3.23 from~\cite{ES}.
\begin{lem}
\begin{enumerate}
\item
The star-product $*$ on $A$ is conjugation-invariant if and only if $\rho$ is a conjugation on $\mc{A}$ that conjugates $T$.
\item
In this situation $*$ is Hermitian if and only if in addition $\rho^2=g$.
\end{enumerate}
\end{lem}
\begin{proof}
\begin{enumerate}
\item
Suppose that $*$ is conjugation-invariant. Then $\rho$ is by definition a conjugation on $\mc{A}$. We have \[T(\rho(a))=T_0(CT(\rho(a))=T_0(\rho(CT(a))=\ovl{(T_0(CT(a))}=\ovl{T(a)}.\]

Suppose that $\rho$ is a conjugation on $\mc{A}$ that conjugates $T$. We want to prove that the corresponding short star-product $*$ on $A$ satisfies\linebreak $\rho_1(a*b)=\rho_1(a)*\rho_1(b)$, where $\rho_1=\gr\rho$. Let $\phi\colon A\to \mc{A}$ be the corresponding quantization map. Since $\phi$ is a linear isomorphism, it is enough to prove that $\phi(\rho_1(a*b))=\phi(\rho_1(a)*\rho_1(b))$. 

In order to do this, it is enough to prove that $\phi\rho_1=\rho\phi$. Indeed, we constructed $*$ so that\[\phi(\rho_1(a)*\rho_1(b))=\phi(\rho_1(a))\phi(\rho_1(b)).\] This is equal to \[\rho(\phi(a))\rho(\phi(b))=\rho(\phi(a)\phi(b))=\rho(\phi(a*b))=\phi(\rho_1(a*b)).\]

Note that the corresponding bilinear form $(a,b)=T(ab)$ satisfies \[(\rho(a),\rho(b))=T(\rho(a)\rho(b))=T(\rho(ab))=\ovl{T(ab)}=\ovl{(a,b)}.\] Since $\rho$ preserves both $\mc{A}_{\le n}$ and $\mc{A}_{\le n-1}$, it preserves $\mc{A}_n:=(\mc{A}_{\le n-1})^{\perp}\cap \mc{A}_{\le n}$. 

By construction we have $\phi(a+\mc{A}_{\le n-1})=a$ when $a\in \mc{A}_n$. Since $\rho_1=\gr\rho$ for any $a\in\mc{A}_n$ we have $\rho(a)+\mc{A}_{\le n-1}=\rho_1(a+\mc{A}_{\le n-1})$. It follows that $\phi^{-1}(\rho(a))=\rho_1(\phi^{-1}(a))$, hence $\phi\rho_1=\rho\phi$ as we wanted.
\item
Suppose that $*$ is Hermitian. By definition this means that $T$ is \linebreak $\rho^2$-twisted. Since $T$ has no kernel, we get $g=\rho^2$.

Suppose that $T$ is a $\rho^2$-twisted trace on $\mc{A}$. Note that $T_0(CT(a))=T(\phi(a))$ for any $a\in A$. It follows that $T_0\circ CT$ is $\rho_1^2$-twisted, as required.
\end{enumerate}
\end{proof}
\subsection{Bimodules and positive definite Hermitian forms}
As in subsection~\ref{SubSecExampleTwoByTwo}, let $A_0=\Mat_2(\CN)$ or $\CN\oplus\CN$, so that $A$ and $\mc{A}$ correspond to Morita contexts. Let us change notation: graded algebra is $\begin{pmatrix}
A & M\\
N & B
\end{pmatrix}$, filtered deformation is $\begin{pmatrix}
\mc{A} & \mc{M}\\
\mc{N}&\mc{B}
\end{pmatrix}$. Assume that there exists antilinear isomorphisms $\rho_1\colon \mc{A}\to\mc{B}$, $\rho_2\colon\mc{B}\to\mc{A}$, $\phi\colon \mc{M}\to\mc{N}$, $\psi\colon \mc{N}\to \mc{M}$ such that \[\rho\begin{pmatrix}
a & m\\
n & b
\end{pmatrix}=\begin{pmatrix}
\rho_2(b) & \psi(n)\\
\phi(m) & \rho_1(a)
\end{pmatrix}\] is a conjugation. For $\mc{M},\mc{N}$ this means that $\phi\colon \mc{M}\to\ovl{\mc{N}}_{\rho}$ is an isomorphism of bimodules, where we use $\rho_1,\rho_2$ to define bimodule structure on $\ovl{\mc{N}}_{\rho}$: $a.n=\rho_1(a)n$, $n.b=n\rho_2(b)$. A similar statement holds for $\psi$. Since $\rho^2$ is an automorphism of the matrix algebra, we get that $\psi\phi$ is an isomorphism between $\mc{M}$ and $\mc{M}$ with bimodule structure twisted by $\rho_1\rho_2,\rho_2\rho_1$ respectively. Abusing notation denote the restriction of $\rho$ to $A\oplus B$ by $\rho$.

Recall that a trace $T$ on the matrix algebra is a pair of traces $T_1$, $T_2$ on $\mc{A}$,$\mc{B}$. A trace $T$ is $\rho^2$-twisted if $T_1$, $T_2$ are $\rho^2$-twisted and \[T_1(mn)=T_2(n\psi\phi(m)),\quad T_2(nm)=T_1(m\phi\psi(n)).\]

Consider the sesquilinear form $(k,l):=T_1(k\phi(l))$ for $k,l\in \mc{M}$. Note that it is invariant in the following sense: \begin{multline*}
(am_1,m_2)=T(am_1\psi(m_2))=T(m_1\psi(m_2)\rho_2\rho_1(a))=\\
T(m_1\psi(m_2\rho_1(a)))=(m_1,m_2\rho_1(a)).
\end{multline*} When this form is positive definite we say that $T$ is positive for bimodules. 

Hence from a matrix algebra and a trace we get an invariant sesquilinear form. On the other hand, suppose that we have algebras $\mc{A},\mc{B}$, an $\mc{A}-\mc{B}$-bimodule $\mc{M}$, an antilinear isomorphism $\rho=\rho_1\oplus \rho_2\colon \mc{A}\oplus\mc{B}\to \mc{B}\oplus\mc{A}$, an isomorphism $g\colon \mc{M}\to \mc{M}_{\rho^2}$ and a sesquilinear form on $\mc{M}$ such that \[(am_1,m_2)=(m_1,m_2\rho(a)),\]
\[(m_1\rho(a),m_2)=(m_1,am_2).\]
We say that such forms are $\rho$-invariant. Note that, strictly speaking, such forms are $\rho_1$-invariant, we do not use $\rho_2$. Below we will get a $\rho_2$-invariant form on $\mc{N}$, so that the picture becomes symmetric.

Define $\mc{N}$ to be $\ovl{\mc{M}}_{\rho^{-1}}$. Then we can define $\phi\colon \mc{M}\to \ovl{\mc{N}}_{\rho}$ to be the identity map on the underlying set and $\psi\colon\mc{N}\to\ovl{\mc{M}}_{\rho}$ to be $g$ on the underlying set. We see that both $\phi,\psi$ are maps of bimodules. Hence we recover the matrix algebra $\begin{pmatrix}
\mc{A} & \mc{M}\\
\mc{N}&\mc{B}
\end{pmatrix}$ and a conjugation.

Recall that for a twisted bimodule $\mc{N}=\ovl{\mc{M}}_{\rho^{-1}}$ we use $b.n$ to mean the action on $\mc{N}$ and $an$ to mean the action on $\mc{M}$, so that $b.n=\rho^{-1}(a)n$.

For $n_1,n_2\in \mc{N}=\ovl{\mc{M}}_{\rho^{-1}}$ define $(n_1,n_2)_{\mc{N}}=(n_2,n_1)_{\mc{M}}$. Then 
\begin{multline*}(b.n_1,n_2)_{\mc{N}}=(\rho^{-1}(b)n_1,n_2)_{\mc{N}}=(n_2,\rho^{-1}(b)n_1)_{\mc{M}}=\\
(n_2b,n_1)_{\mc{M}}=(n_1,n_2b)_{\mc{N}}=(n_1,n_2.\rho(b)),
\end{multline*} hence $(\cdot,\cdot)_{\mc{N}}$ is also $\rho$-invariant, but here we use $\rho_2$, not $\rho_1$.

 We turn to the traces:


\begin{prop}
\label{PropFormsAndTraces}
\begin{enumerate}
\item
$\rho$-invariant sesquilinear forms on $\mc{M}$ are in one-to-one correspondence with $\rho^2$-twisted traces $\mc{M}\otimes_{\mc{B}}\mc{N}\to \CN$. Each invariant form also gives a $\rho^2$-twisted trace on $\mc{N}\otimes_{\mc{A}}\mc{M}$.
\item
Suppose that $\mc{M}$ and $\mc{N}$ provide a Morita equivalence between $\mc{A}$ and $\mc{B}$. Then there is one-to-one correspondence between $\rho$-invariant sesquilinear forms on $\mc{M}$ and $\begin{pmatrix}
\rho^2 & g\\
h & \rho^2
\end{pmatrix}$-twisted traces on $\begin{pmatrix}
\mc{A} & \mc{M}\\
\mc{N}& \mc{B}
\end{pmatrix}$. Here $h$ is a unique isomorphism from $\mc{N}$ to $\mc{N}_{\rho^2}$ such that $g\otimes h=\rho^2$.
\end{enumerate}
\end{prop}
\begin{proof}
\begin{enumerate}
\item
Let $(\cdot,\cdot)$ be a $\rho$-twisted form. Consider the map $T\colon \mc{M}\otimes_{\mc{B}}\mc{N}\to \CN$ given by $T(m\otimes n)=(m,n)$. We have \[T(mb\otimes n)=(mb,n)=(m,\rho^{-1}(b)n)=T(m\otimes b.n),\] hence $T$ is well-defined. We also have \begin{multline*}
T(am_1\otimes m_2)=(am_1,m_2)=(m_1,m_2\rho(a))=\\
(m_1,m_2.\rho^2(a))=T((m_1\otimes m_2)\rho^2(a))
\end{multline*} In other words, $T$ is a $\rho^2$-twisted trace.

In the other direction, if $T$ is a $\rho^2$-twisted trace we can define $(\cdot,\cdot)$ using $(m,n)=T(m\otimes n)$. The same computations as above show that $(\cdot,\cdot)$ is $\rho$-invariant.

If we use the same reasoning for $\mc{N}$ instead of $\mc{M}$ we obtain a $\rho^2$-twisted trace on $\mc{N}\otimes_{\mc{B}}\mc{N}_{\rho^{-1}}$. Note that $\mc{N}_{\rho^{-1}}=\mc{M}_{\rho^{-2}}$. Applying $\id_{\mc{N}}\otimes g$ we get $\mc{N}\otimes\mc{N}_{\rho^{-1}}\cong \mc{N}\otimes\mc{M}$. Hence the second trace $T'$ is given by \[T'(n\otimes m)=(n,g^{-1}m)_{\mc{N}}=(g^{-1}m,n).\]
\item
Since $\mc{M}\otimes_{\mc{A}}\mc{N}$ is isomorphic to $\mc{A}$, there is one-to-one correspondence between invariant forms on $\mc{M}$ and twisted traces on $\mc{A}$, similarly for $\mc{B}$. It remains to prove that every such pair of traces $T_1$, $T_2$ on $\mc{A}$, $\mc{B}$, satisfies \[T_1(mn)=T_2(ng(m)),\quad T_2(nm)=T_1(mg(n)).\] Indeed, $T_1(mn)=(m,n)=T_2(ng(m))$ and \[T_2(nm)=(g^{-1}m,n)=T_1(g^{-1}(m)n)=T_1(mh(n)).\] On the last step we used that $g\otimes h=\rho^2$ and $T_1$ is $\rho^2$-twisted.
\end{enumerate}
\end{proof}

In the case when $\mc{A}=\mc{B}=\mc{M}$, $\rho_1=\rho_2=\rho$ we have just one $\rho^2$-invariant trace $T$ on $\mc{A}$. In this case $\mc{N}=\mc{A}$, the corresponding isomorphism of bimodules sends $a\in\mc{N}$ to $\rho(a)\in\mc{A}$. We have $\mc{M}\otimes_{\mc{B}}\mc{N}=\mc{A}$, so the corresponding trace is a $\rho^2$-twisted trace on $\mc{A}$.  The tensor product $\mc{M}\otimes_{\mc{A}} \mc{M}_{\rho^{-1}}$ is isomorphic to $\mc{A}$ via $a\otimes b\mapsto a\rho(b)$, so the positivity condition is $T(a\rho(a))>0$. We recover the positivity condition from~\cite{EKRS} in this case.

In the case we are interested in, namely $\mc{A},\mc{B}$ are ($q$-)deformations of Kleinian singulartieis, bimodule positivity implies that the corresponding traces on $\mc{A}$, $\mc{B}$ are nondegenerate. It follows that the trace on the matrix algebra is also nondegenerate. Using Proposition~\ref{PropStarProductsAndTwistedTraces} we obtain a star-product on the algebra $\begin{pmatrix}
A & M\\N & B
\end{pmatrix}$.

\section{Unitarizability in the case of generic parameter.}
\label{SecUsualUnitarizability}
\subsection{Quantizations and bimodules via differential operators}
Let $c=(c_1,\ldots,c_n)$ be a sequence of complex numbers. When none of $c_i-c_j$, $i\neq j$, are integers, we have
\begin{lem}[\cite{BEF},Proposition 2.11]
Consider the algebra $A_c$ of differential operators with a pole at only possibly zero that preserve each set $x^{c_i}\CN[x]$, $i=1,\ldots,n$. This algebra is generated by $v=x$, $z=E=x\partial_x$ and $u=x^{-1}(E-c_1)\cdots(E-c_n)$ and the defining set of relations is $[z,u]=-u$, $[z,v]=v$, $uv=P(z+\tfrac12)$, $vu=P(z-\tfrac12)$, where $P(t)=\prod_{i=1}^n (t-c_i+\tfrac12)$.
\end{lem}

The proof of this proposition can be found in~\cite{BEF}. Also Lemma~\ref{LemQAlgOfDiffOp} will have a similar proof.

We say that $c$ is a generic parameter if none of $c_i-c_j$, $i\neq j$ are integers.

\begin{rem}
For non-generic parameters the lemma should be modified as follows. If we have $c_1\leq c_2\leq \cdots\leq c_l$ a sequence of parameters with integer difference. Then the elements $u,v,z$ defined as in the lemma, preserve each of the sets $x^{c_i}\C[x]$,  $x^{c_{i-1}}\C[x]\oplus x^{c_i}\C[x]\ln x$, and so on until \[x^{c_1}\C[x]\oplus x^{c_2}\C[x]\ln x\oplus\cdots x^{c_i}(\ln x)^{i-1}\C[x].\] Then the set of differential operators with a pole at only possibly zero that preserve each of these sets for all $i$ from $1$ to $n$ should coincide with an algebra generated by $u,v,z$. We do not need the lemma in this generality since we use genericity condition several times below.
\end{rem}

Suppose that $c$ and $c'$ are two generic parameters such that $c_i-c'_i$ are integers. Consider the $A_c$-$A_{c'}$-bimodule $M=M_{c,c'}$ that consists of differential operators with pole at only possibly zero that send each $x^{c'_i}\CN[x]$ to $x^{c_i}\CN[x]$.

\begin{lem}
\label{LemHowMCCPrimeLooks}
\begin{enumerate}
\item
The module $M_{c,c'}$ is the direct sum of its $\ad z$-eigenspaces. The eigenspace of weight $j$ is $x^j R_j(z)\CN[z]$, where $R_j$ is a monic polynomial with simple zeroes at $c'_i,\ldots,c_i-j-1$ for all $i$ such that $c_i-c'_i-j>0$.
\item
$M_{c,c'}$ is a Harish-Chandra bimodule in the sense of~\cite{LosevHarishChandra}: there exists an increasing filtration $M_{\le j}$ with $\cup_{j} M_{\le j}=M_{c,c'}$ compatible with filtrations on $A_c,A_{c'}$ with the following properties:
\begin{itemize}
\item
For each $j$ the subspace $M_{\le j}$ is preserved by $\ad z$ and $[u,M_{\le j}]\subset M_{\le n+j-2}$, $[v,M_{\le j}]\subset M_{\le n+j-2}$.
\item
The previous condition makes $\gr M_{c,c'}$ into $\gr A_c\cong \CN[x,y]^{C_n}\cong \gr A_{c'}$-bimodule such that the action on the left coincides with the action on the right. Then $\gr M_{c,c'}$ is a finitely generated $\gr A_c$-module.
\end{itemize}
\end{enumerate}
\end{lem}
\begin{proof}
\begin{enumerate}
\item
For $a,b\in\CN$, $m=\sum x^i S_i(z)\in \CN[x,x^{-1},\partial_x]$ we have the following: $mx^a$ belongs to $x^b\CN[x]$  if and only if each $x^i S_i(z)x^a$ belongs to $x^b\CN[x]$. It follows that $m=\sum x^i S_i(z)$ belongs to $M_{c,c'}$ if and only if each $x^i S_i(z)$ belongs to $M_{c,c'}$. This proves the first statement. 

Consider $x^j R(z)\in M_{c,c'}$. This element sends $x^{c'_i+k}$ to $x^{c'_i+k+j}R(c'_i+k)$. Hence when $k\geq 0$ and $c'_i+k+j<c_i$ we should have $R(c'_i+k)=0$. It follows that $R$ has roots $c'_i+k$ for all $k$ such that $0\le k<c_i-c'_i-j$. This proves the second statement.
\item
Denote  $M:=M_{c,c'}$ and denote by $M_j$ is an $\ad z$ eigenspace of $M$ of weight $j$. 


Let $M_{\le k}$ consist of all $x^j R(z)$ such that $nj+2\deg R\le k$. Note that for negative $j$ with $\abs{j}$ large enough the minimal possible degree of $R$ is $-nj$ plus some constant. This proves that $M_{\le k}$ is finite-dimensional for all $k$ and empty for large negative $k$. We see that this filtration is compatible with filtrations on $A_c,A_{c'}$ and each $M_{\le k}$ is preserved by $\ad h$. In order to check the adjoint action let us take $x^j R(z)\in M_{\le k}$. Then
\[[v,x^j R(z)]=x^{j+1}\big(R(z)-R(z+1)\big),\]
\[[u,x^j R(z)]=x^{j-1}\big(P(z+j)R(z)-P_1(z)R(z-1)\big),\]
where $P_1$ is a polynomial for algebra $A_{c'}$. Both $R(z)-R(z+1)$ and $P(z+j)R(z)-P_1(z)R(z-1)$ are differences of two polynomials with the same leading coefficient, hence their degrees are at most $\deg R-1$, $\deg R+\deg P-1=\deg R+n-1$ respectively. We see that both $[u,x^j R(z)]$ and $[v,x^j R(z)]$ belong to $M_{\le k+n-2}$.

Let us compute the action of $\gr A_c=\CN[x,y]^{C_n}=\CN[u,v,z]/(uv-z^n)$ on $\gr M$. For $a\in M_{\le k}\setminus M_{\le k-1}$ denote by $\ovl{a}$ the corresponding element of $(\gr M)_k$. Take $\ovl{a}\in \gr M$. We can assume that $a=x^j z^l$ with $z+2l=k$. Then \[v\ovl{a}=\ovl{va}=\ovl{x^{j+1}z^l},\]
\[z\ovl{a}=\ovl{za}=\ovl{zx^jz^l}=\ovl{(x+1)^j z^{l+1}}=\ovl{x^j z^l},\]
\[u\ovl{a}=\ovl{ua}=\ovl{x^{-1}P(z-\tfrac12)x^jz^l}=\ovl{x^{j-1}P(z+j-\tfrac12)z^l}=\ovl{x^{j-1}z^{n+l}}.\]
This means that $\gr M$ is $\CN[x,y]^{C_n}$-submodule of $\CN[x,x^{-1},y]$ with $\ovl{x^j z^l}\in \gr M$ corresponding to $x^{nj+l}y^l\in \CN[x,x^{-1},y]$.

Note that for large enough $j\geq j_0$ we have $R_j(z)=1$ and \[R_{-j-1}(z)=(z-c_1-j)(z-c_2-j)\cdots (z-c_n-j)R_{-j}(z).\]  We see that $\gr M$ is generated by $\ovl{x^j R_j(z)}$ for all $j$ with $\abs{j}\leq j_0$.


\end{enumerate}
\end{proof}
\begin{rem}
In the case $n=2$ algebras $A_c$, $A_{c'}$ are central reductions of $U(\mf{sl}_2)$ and our definition of Harish-Chandra bimodule agrees with the standard one: the adjoint action of $U(\mf{sl}_2)$ is locally finite. In the case $n>2$ we use generators $u,v,z$ to write a definition similar to the one that Losev gives in~\cite{LosevHarishChandra}: $\mc{M}$ is a filtered bimodule over a filtered algebra $\mc{A}$, $d$ is a positive integer such that $[\mc{A}_{\le i},\mc{M}_{\le j}]\subset \mc{M}_{\le i+j-d}$, module $\gr\mc{M}$ is finitely generated over $\gr\mc{A}$. Our definition is equivalent to Losev's if we take $\mc{A}$ to be the algebra with generators $u,v,z$ and relations $[z,v]=v$, $[z,u]=-u$, so that it has both $A_c$ and $A_{c'}$ as its quotients, $\mc{M}=M_{c,c'}$ and $d=2$. Note that there is no relation on $[u,v]$ because $A_c$ and $A_{c'}$ have different expressions of $[u,v]$ as a polynomial in $z$.
\end{rem}
\begin{lem}
\label{LemMoritaEquivalence}
$M_{c,c'}$ and $M_{c',c}$ give a Morita equivalence between $A_c$ and $A_{c'}$.
\end{lem}
\begin{proof}
Let $U$ be any of $A_c,A_{c'},M_{c,c'},M_{c',c}$. By definition $U\subset k[x,x^{-1},\partial_x]$. We also see that $U$ is a $\CN[v]=\CN[x]$-module and for any $p\in \CN[x,x^{-1},\partial_x]/U$ there exists $k$ such that $x^kp=0$. It follows that $k[x,x^{-1}]\otimes_{k[x]}U$ is isomorphic to $k[x,x^{-1},\partial_x]$. 

Consider the map $\phi\colon M_{c,c'}\otimes_{A_{c'}} M_{c',c}\to A_c$ that sends $f\otimes g$ to $fg$. It can be proved that $M_{c,c'}\otimes_{A_{c'}}M_{c',c}$ does not have $x$-torsion. Alternatively, we can use the fact that it is enough to prove surjectivity in Morita context. Since $k[x,x^{-1}]$ is a flat $k[x]$-module it is enough to prove that $\psi=\id_{k[x,x^{-1}]}\otimes \phi$ is an isomorphism. After identifying $k[x,x^{-1}]\otimes_{k[x]}M_{c,c'}$ and $k[x,x^{-1}]\otimes_{k[x]}A_c$ with $k[x,x^{-1},\partial_x]$ we get $\psi\colon k[x,x^{-1},\partial_x]\otimes_{A_{c'}}M_{c',c}\to k[x,x^{-1},\partial_x]$ given by $f\otimes g\mapsto fg$.

We have $k[x,x^{-1},\partial_x]=k[x,x^{-1}]\otimes_{k[x]}A_{c'}$. Using this we get \[k[x,x^{-1},\partial_x]\otimes_{A_{c'}}M_{c',c}=k[x,x^{-1}]\otimes_{k[x]}M_{c',c}=k[x,x^{-1},\partial_x]\] and $\psi$ becomes identity. 

Hence $\phi$ is an isomorphism. We similarly prove the similar map from $M_{c',c}\otimes_{A_c} M_{c,c'}$ to $A_{c'}$ is an isomorphism. The lemma follows.
\end{proof}
\begin{cor}
\label{CorHCEquivalence}
The category of Harish-Chandra $A_c$-$A_{c'}$ bimodules is equivalent to the category of Harish-Chandra $A_c$-bimodules.
\end{cor}

We need the following proposition for completeness, but we do not use it anywhere below.
\begin{prop}
\label{PropHCEverything}
The category of Harish-Chandra $A_c$-$A_{c'}$-bimoodules is semisimple. The simple objects in this category can be obtained as $M_{c,c''}$, where $c''$ is a parameter such that the quantization $A_{c''}$ is isomorphic to $A_{c'}$.
\end{prop}
\begin{proof}
Using Corollary~\ref{CorHCEquivalence} it is enough to prove that in the case when $c=c'$. In this case the category of Harish-Chandra bimodules was described by Simental in~\cite{Jose}. We will use the description from Losev's article~\cite{LosevHC}, namely Theorem~1.2. Losev describes the quotient of the category of Harish-Chandra bimodules by the subcategory of bimodules with support of non-maximal dimension. 

In our case this subcategory is trivial. Indeed, support defines a Poisson subscheme of $\CN[u,v]^{\Gamma}$. All such proper subschemes are supported at zero, so that the subcategory consists of finite-dimensional bimodules. The condition on the parameters means that there are no non-zero finite-dimensional submodules. Finite-dimensional representations of more general algebras, $W$-algebras, were classified by Losev~\cite{LosevW}. 

For deformations of Kleinian singularities of type $A$, this can be proved directly as follows. Let $M$ be a finite-dimensional irreducible representation of $A_c$. Let $m\in M$ be an eigenvector of $z$ with eigenvalue $\lambda$. Then $u^k m$ is an eigenvector of $z$ with eigenvalue $\lambda-k$, hence we can find $n\in M$ such that $un=0$ and $zn=\mu n$. Since $vu=P(z-\tfrac12)$, we should have $\mu=c_j$ for some $j$. We can also find $r\in M$ such that $vr=0$ and $zr=\mu+l$ for some nonnegative integer $l$. Using $uv=P(z+\tfrac12)$ we have $\mu+l+1=c_k$ for some $k$. Hence $c_j+l+1=c_k$, this contradicts our choice of parameter $c$.


In~\cite{LosevHC} Losev uses two ways of parametrizing deformations of type $A$ Kleinian singularities.The first one is Crawley--Boevey---Holland construction. The parameter is an element of $Z(\CN[\Gamma])$ of the form $C_{CBH}=1+\sum_{\gamma\neq 1}C_{\gamma}\gamma$. 

CBH parameters $C_{\gamma}$ are expressed in terms of $c$ as follows. For cyclic $\Gamma$ with generator $\gamma$ we denote $C_{\gamma^k}$ by $C_k$. Shifting all $c_i$ such that $\sum c_i=0$ we get \[c_k=\frac{1}n(\sum_{i=1}^{n-1}\frac{C_i \eps^{-ik}}{\eps^i-1}+\frac{1}{2}-k)+\frac{1}{2}.\] For the proof see Lemma~1.2.4 with $q=0$ and $a=n$ in~\cite{KV}, for example.

The second parameter is $\lambda_c\in\mf{h}^*$, it is defined via $\langle \lambda_c,\alpha_k^{\vee}\rangle=tr_{N_k}(C)=C_0+\sum_{j=1}^{n-1}e^{\frac{2\pi \mathrm{i} k j}{n}}C_j$. Here $\alpha_k^{\vee}$ is the $k$-th simple coroot $(0,\ldots,1,-1,\ldots,0)$.

Note that $c_k-c_{k+1}=\frac{1}{n}(\sum_{i=1}^{n-1}\frac{C_i \eps^{-ik}-\eps^{-i(k+1)}}{\eps^i-1})=\frac{1}{n}\sum C_i\eps^{-ik}=tr_{N_{-k}}(C)$.

Considering $c$ as an element of $\mf{h}^*\cong \C^n/\C_{diag}$, we get $\langle c,\alpha_k^{\vee}\rangle=\langle \lambda_c,\alpha_{-k}^{\vee}\rangle$. Hence $c$ and $\lambda_c$ differ by a Dynkin diagram automorphism. Hence the affine Weyl group orbit of $c$ is obtained from affine Weyl group orbit of $\lambda_c$ by a diagram automorphism.





Theorem~1.2 in~\cite{LosevHC} says that the category of Harish-Chandra bimodules over $A_c$ is isomorphic to the category of $\Gamma/\Gamma_0$-representations, where $\Gamma_0$ is the smallest normal subgroup of $\Gamma$ such that there exists $C_0\in \CN[\Gamma_0]$ for which the corresponding parameter $\lambda_0\in \mf{h}^*$ lies in the affine Weyl group orbit of $\lambda$. 

Assume that $A_c$ corresponds to $C_{CBH}\in \C[\Gamma_0]$. This means that $A_c$ is obtained as $\Gamma/\Gamma_0$-invariants of a deformation $\mc{A}$ of $\CN[x,y]^{\Gamma_0}$, see Corollary 2.9 of~\cite{LosevHC}. There is a proof of this statement in Proposition~9.7 of~\cite{IMRNTHesis}. Suppose that $\Gamma_0$ has order $m$. Comparing the relations $uv=P(z-\tfrac12)$ for $A_c$ and for $\mc{A}$ and using $u_c=u_{\mc{A}}^{\frac{n}m}$, $v_c=v_{\mc{A}}^{\frac{n}m}$, $z_c=\frac{m}{n}z_{\mc{A}}$ we get $c_{k+m}=c_k+\frac{m}{n}$ for all $k$. 

We have to describe $\frac{n}{m}$ distinct nontrivial Harish-Chandra bimodules over $A_c$. Let $0\leq i<\frac{n}{m}$. For $1\leq k\leq im$ let $c'_k=c_k+1$. Note that $c'_k=c_{k-im}+\frac{im}{n}$: if $k>im$, then $c'_k=c_k=c_{k-im}+\frac{im}{n}$ and if $k\leq im$ then \[c'_k=c_k+1=c_{k+(\frac{n}{m}-i)m}-(\frac{n}{m}-i)\frac{m}{n}+1.\] Hence $A_{c'}$ is isomorphic to $A_c$. 

We get an $A_c-A_{c'}$ bimodule $M_i=M_{c,c'}$. To show that these bimodules are not isomorphic to each other for different choices of $i$, we compute the adjoint action of $z$ on $M_s$. The isomorphism between $A_c$ and $A_{c'}$ sends $z\in A_c$ to $z+\frac{im}{n}$. Hence the adjoint action of $z$ on $M_{c,c'}$ has weights in $-\frac{im}{n}+\ZN$. These sets are disjoint for different choices of $0\leq i<\frac{n}{m}$.

Since there are no Harish-Chandra $A_c$-bimodules with support of non-maximal dimension and $\gr A_c$ is a domain, the algebra $A_c$ is simple. Then $M_{c,c'}$ provide a Morita equivalence between two simple algebras, hence these modules are also simple. So, we found the required number of pairwise non-isomorphic simple Harish-Chandra $A_c$-bimodules.

It remains to deal with the case when $\lambda$ is obtained from $\lambda_0$ by an action of the affine Weyl group. Since $c$ and $\lambda$ differ by a diagram automorphism, this means that $c$ is obtained from $c^0$ by an action of the affine Weyl group, where $c^0$ satisfies $c^0_{k+m}=c^0_k+\frac{m}{n}$. Changing the order of $c_1,\ldots,c_n$ if necessary, we can assume that $c_i=c_0^i+l_i$ for some integers $l_1,\ldots,l_n$. 

Now, for each $0\leq s<\frac{n}{m}$, take the corresponding bimodule $M_{c^0,c^1}$. Here $c^1$ is obtained from $c^0$ by an integer shift as above, and there exist a permutation $\pi$ such that $c^0_{\pi(j)}-\frac{sm}{n}=c^1_j$. Let $c'_j=c^0_{\pi(j)}-\frac{sm}{n}+l_{\pi(j)}$. This equals to $c_{\pi(j)}-\frac{sm}{n}$, hence $A_{c'}$ is isomorphic to $A$. On the other hand, $c'_j=c_1^j+l_{\pi(j)}$, hence its entries are integer shifts of the corresponding entries of $c^0$ or, equivalently, $c$. We get a Harish-Chandra $A_c$-bimodule $M_{c,c'}$. As above, the isomorphism between $A_c$ and $A_{c'}$ sends $z$ to $z+\frac{im}{n}$, hence the $\ad z$ weight spaces of $M_{c,c'}$ are disjoint for different values of $s$ and we get $\frac{n}{m}$ non-isomorphic Harish-Chandra bimodules. As above, the algebra $A_c$ is simple and $M_{c,c'}$ are simple bimodules.

\end{proof}
\begin{rem}
Our proof shows that the category of Harish-Chandra bimodules depends only on the extended affine Weyl group orbit of a parameter $c$, because being able to shift each $c_i$ by an integer gives a weight lattice action, not a root lattice action. The description in~\cite{LosevHC} depends on the affine Weyl group orbit. The difference is explained as follows. Suppose that there are $k$ simple Harish-Chandra bimodules. Shifting $c$ by a constant and rearranging we can assume that $c_1=\frac{1}{k}$, $c_2=\frac{2}{k}$, $\ldots$, $c_k=1$. Then $(1,0,\ldots,0)+c$ gives $c_1=\frac{k+1}{k}$, $c_2=\frac{2}{k}$, $\ldots$, $c_k=1$, also an arithmetric progression of length $k$ with difference $\frac{1}{k}$. Hence $(1,0,\ldots,0)+c$ is also a parameter with $k$ nontrivial Harish-Chandra bimodules. Similarly, $(1,1,\ldots,1,0,\ldots,0)+c$ is a parameter with $k$ nontrivial Harish-Chandra bimodules. The elements $(1,1,\ldots,1,0,\ldots,0)$ form a complete set of representatives for the root lattice action on weight lattice.
\end{rem}

\subsection{Isomorphism between $M$ and $M_{\rho}$}

Now we assume that $A_{c'}$ is isomorphic to $\ovl{A_c}$ and that both maps $A_c\to \ovl{A_{c'}}$, $A_{c'}\to \ovl{A_c}$ are given by $v\mapsto au$, $u\mapsto bv$, $z\mapsto -z$. Abusing notation we denote both maps by $\rho$. 

 It follows from discussion in Section~2.3 in~\cite{EKRS} that we may take $a,b$ such that $\abs{a}=1$ and $ab=(-1)^n$. Hence $a=\pm i^n e^{-\pi i c}$, $b=\pm i^n e^{\pi i c}$. These isomorphisms are well-defined when $P_{c'}(x)=(-1)^n\ovl{P_c}(-x)$. Both sides have the same leading coefficient, so this is equivalent to having the same set of roots. We get the following condition: for any $i$ from $1$ to $n$ there exists $j$ such that $c_i-\tfrac12=\tfrac12-\ovl{c'_j}$. The latter is equivalent to $c_i+\ovl{c'_j}=1$.
 
Suppose that $j$ corresponds to $i$ and $k$ corresponds to $j$: $c_i+\ovl{c'_j}=1$, $c_j+\ovl{c'_k}=1$. Conjugating the second equation and subtracting we get \[c_i+\ovl{c'_j-c_j}-c'_k=0,\] hence \[c_i-c_k=\ovl{c_j-c'_j}+c'_k-c_k\] is a sum of two integers. This contradicts our assumption that $\{c_1,\ldots,c_n\}$ is a generic parameter. 

Hence numbers from $1$ to $n$ are divided into pairs $(i,j)$ and singletons $i=j$ such that $c_i+\ovl{c'_j}=c_j+\ovl{c'_i}=1$.

Proposition~\ref{PropFormsAndTraces} says that $\rho$-invariant forms on $M_{c,c'}$ are in one-to-one correspondence with $\rho^2$-twisted traces on $M_{c,c'}\otimes_{A_{c'}}M_{c,c',\rho^{-1}}$. Here for an $A_{c}-A_{c'}$-bimodule $M$ by $M_{\rho}$ we mean $\ovl{M}$ with the action $b.m=\rho(b)m$, $m.a=m\rho(a)$ for $a\in A_c$,$ b\in A_{c'}$, $m\in \ovl{M}$. We want $A_c,A_{c'}, M_{c,c'},M_{c',c}$ to form Morita context with conjugation as in example in Section~\ref{SubSecConjugations}. The two remaining pieces are isomorphisms $M_{c,c'}\cong M_{c',c,\rho}$ and $M_{c',c}\cong M_{c,c',\rho}$. We can interchange $c$ and $c'$, so it is enough to find just one of these two isomorphisms.

The bimodule $M_{c,c'}$ is Harish-Chandra. This shows why we want the action of $\rho$ on generators $u,v,z$ to be the same for $A_c$ and $A_{c'}$: if $M_{c',c',\rho}$ is isomorphic to $M_{c,c'}$, it is also Harish-Chandra. Hence, for example, the action $m\mapsto \rho_1(v)m-m\rho_2(v)$ should send $M_{\le k}$ to $M_{\le k+n-2}$. This is possible only when $\rho_1(v)$ and $\rho_2(v)$ are the same multiple of $u$. Similarly, $\rho(z)$, $\rho(u)$ should be the same for $A_c$, $A_{c'}$.
\begin{lem}

\label{LemCCprimeRhoIsomorphic}
The map $\phi$ given by $\phi(x^j R_j(z)R(z))=x^{-j}S_{-j}(z)\ovl{R}(-h)$ is an isomorphism from $M_{c,c'}$ to $M_{c',c,\rho}$. The map $\phi$ also gives an isomorphism from $M_{c,c',\rho^{-1}}$ to $M_{c',c}$. Here for an integer $j$ and all $i$ such that $c'_i-j-1\geq c_i$ the polynomial $S_j$ has roots $c_i,\ldots,c'_i-j-1$. The leading coefficient of $S_j$ is $C_{\phi}(-1)^{\deg R_{-j}}a^{-j}$, where $C_{\phi}$ is a constant corresponding to a choice of $\phi$.
\end{lem}
\begin{proof}
The second statement follows from the first after twisting the action from both sides by $\rho^{-1}$.

We want to construct a linear isomorphism $\phi\colon M_{c,c'}\to M_{c',c}$ such that $\phi(em)=v.\phi(m)=au\phi(m)$, $\phi(um)=u.\phi(m)=bv\phi(m)$, $\phi(zv)=z.\phi(v)=-z\phi(v)$, similarly for the right multiplication.

We define $\phi(x^j R_j(z)R(z)):=x^{-j}S_{-j}(z)\ovl{R}(-h)$. Here $S_{-j}$ is a polynomial that satisfies a similar condition on roots as $R_j$ but is not necessarily monic. Namely, $S_j$ has roots $c_i,\ldots,c'_i-j-1$ for all $i$ such that $c'_i-j-1\geq c_i$.

We see that $\phi$ is linear and satisfies $\phi([z,m])=[\phi(m),z]$ and $\phi(mz)=-\phi(m)z$ for all $m\in M_{c,c',\rho}$. Hence $\phi(zm)=-z\phi(m)$ for all $m\in M_{c,c',\rho}$.

It is enough to check all other conditions for $m=x^j R_j(z)$.

We have 
\begin{equation}
\label{EqRho1}
\phi(vm)=\phi(x^{j+1}R_j(z))=\phi(x^{j+1}R_{j+1}(z)L_j(z))=x^{-j-1}S_{-j-1}(z)\ovl{L_j}(-h).
\end{equation} Here $L_j(t)=\frac{R_j(t)}{R_{j+1}(t)}$ is a monic polynomial with roots $c_i-j-1$ for all $i$ such that $c_i-j-1\geq c'_i$. 

Let $\sigma$ denote the permutation such that $\ovl{c_i}+c'_{\sigma(i)}=1$, $\sigma^2=1$. Suppose that $c_i-j-1\geq c'_i$. We get $\ovl{c_i}-j-1\geq \ovl{c'_i}$, hence $1-c'_{\sigma(i)}-j-1\geq 1-c_{\sigma(i)}$. It follows that $c_{\sigma(i)}-j-1\geq c'_{\sigma(i)}$. We see that roots of $L$ come in pairs $c_i-j-1$, $c_{\sigma(i)}-j-1$.

The polynomial $\ovl{L_j}(-t)$ has roots $-(\ovl{c_{i-j-1}})=j+1-\ovl{c_i}=j+c'_{\sigma(i)}$ for all $i$ such that $c_i-j-1\geq c'_i$. Taking $i$ instead of $\sigma(i)$, the roots become $j+c'_i$.

We have $S_{-j-1}(t)=S_{-j}(t)M_{-j-1}(t)$, where $M_{-j-1}$ has roots $j+c'_i$ for all $i$ such that $c'_i+j\geq c_i$. Similarly to the above, roots of $M_{-j-1}$ are in pairs $j+c'_i$, $j+c'_{\sigma(i)}$.

For every $i$ either $c_i-j-1\geq c'_i$ or $c'_i+j\geq c_i$ is true but not both. We deduce that the union of roots of $M_{-j-1}$ and $\ovl{L_j}(-t)$ is disjoint and equal to $\{c'_1+j,\ldots,c'_n+j\}$.

Using~\eqref{EqRho1} we get
\begin{equation}
\label{EqRho2}
\phi(vm)=x^{-j-1}S_{-j-1}(z)\ovl{L_j}(-z)=x^{-j-1}S_{-j}(z)M_{-j-1}(z)\ovl{L_j}(-z).
\end{equation}

From $M_{-j-1}=\frac{S_{-j-1}}{S_{-j}}$ we deduce that $M_{-j-1}$ has leading coefficient \[\frac{(-1)^{\deg R_{j+1}}a^{j+1}}{(-1)^{\deg R_j}a^j}=(-1)^{\deg R_{j+1}-\deg R_j}a=(-1)^{\deg L_j}a.\] Hence $M_{-j-1}(z)\ovl{L_j}(-z)$ has leading coefficient $a$ and $M_{-j-1}(z)\ovl{L_j}(-z)=a(z-c'_1-j)\cdots (z-c'_n-j)=P_1(z-\tfrac 12-j)$. We deduce from~\eqref{EqRho2} that $\phi(vm)=ax^{-j-1}S_{-j}(z)P_1(z-\tfrac12-j)$.

We have \[au\phi(m)=a x^{-1}P_1(z-\tfrac12)x^{-j}S_{-j}(z)=ax^{-j-1}P_1(z-\tfrac12-j)S_{-j}(z)=\phi(vm).\]

It follows from the description of $M_{c,c'}$ that it is a torsion-free $A_c$ and $A_{c'}$ module. Hence in order to check that $\phi(um)=bv\phi(m)$, we can check that $u\phi(um)=buv\phi(m)$. This follows from what we already proved:
\begin{multline*}
buv\phi(m)=bP_1(z+\tfrac12)\phi(m)=b\phi(\ovl{P_1}(\tfrac12-z)m)=\\
(-1)^nb\phi(P(z-\tfrac12)m)=(-1)^nb\phi(vum)=(-1)^nabu\phi(um)=u\phi(um).
\end{multline*}
We used that $\ovl{P_1}(-t)=(-1)^nP(t)$, $ab=(-1)^n$ and $\phi(vum)=au\phi(um)$.

Similarly the condition $\phi(mu)=\phi(m)bv$ will follow from the condition $\phi(mv)=\phi(m)au$ for all $m$. We will prove it now. We have \begin{multline}
\label{EqRho3}
\phi(mv)=\phi(x^j R_j(z)x)=\phi(x^{j+1}R_j(z+1))=\\
\phi(x^{j+1}R_{j+1}(z)A_j(z))=x^{-j-1}S_{-j-1}(z)\ovl{A_j}(-z).
\end{multline}
Here $A_j$ has roots $c'_i-1$ for all $i$ such that $c_i-j-1\geq c'_i$ and $\ovl{A_j}(-z)$ has roots $-(\ovl{c'_i}-1)=c_{\sigma(i)}$ for the same $i$. Reasoning as with $L$ we see that the roots of $A_j$ come in pairs $c_i,c_{\sigma(i)}$.

We have $S_{-j-1}(t)=S_{-j}(t-1)B_{-j-1}(t)$. The polynomial $B_{-j-1}$ has roots $c_i$ for all $i$ such that $c'_i+j\geq c_i$. Reasoning as above, we deduce that the union of the roots of $A_j(t)$ and $B_{-j-1}(t)$ is disjoint and equals to $\{c_1,\ldots,c_n\}$.

Similarly to the above, we deduce that $B_{-j-1}(z)\ovl{A_j}(-h)=aP(z-\tfrac12)$. The leading signs coincide because the leading sign of $B_{-j-1}$ equals to the leading sign of $M_{-j-1}$ and the degree of a monic polynomial $A_j$ equals to the degree of a monic polynomial $L_j$.


Combining this and~\eqref{EqRho3} we get $\phi(mv)=ax^{-j-1}S_{-j}(z-1)P(z-\tfrac12)$.

We have \[\phi(m)au=ax^{-j}S_{-j}(z)x^{-1}P(z-\tfrac12)=ax^{-j-1}S_{-j}(z-1)P(z-\tfrac12).\]
 The lemma follows.
\end{proof}

\subsection{The positive forms}
Proposition~\ref{PropFormsAndTraces} says that any $\rho$-invariant sesquilinear form $(\cdot,\cdot)$ on $M$ is given by $(m,n)=T(m\phi(n))$ , where $T$ is a $g_t=\rho^2$-twisted trace. Here $t=ba^{-1}$. For fixed $t$ there are two conjugations $\rho$ with $\rho^2=g_t$. One of them is $\rho_+$, the other is $\rho_-$, later we will specify which is which. The answer for $\rho_+$ and $\rho_-$ is sometimes different, as it was in~\cite{EKRS}.

We want to describe all traces in convenient form. Recall that $\abs{t}=1$. Let $t=e^{2\pi i c}$, where $c\in[0,1)$. We will need the following definition:
\begin{defn}
We say that a non-self-intersecting curve $C$ on a complex plane is an good contour if the following holds:
\begin{enumerate}
\item
There exists $r>0$ such that $C\setminus B_r(0)$ coincides with $(a+i\RN)\setminus B_r(0)$ for some $a\in\RN$. This allows us to define the notions ''to the left of $C$'' and ''to the right of $C$''.
\item
For every $i=1,\ldots,n$ the set $c_i-\ZN_{> 0}$ is to the left of $C$ and $c_i+\ZN_{\geq 0}$ is to the right of $C$. 
\end{enumerate}
\end{defn}
We note that for generic $c=\{c_1,\ldots,c_n\}$ there exist good contours.

Let $\bP(x)=\prod_{i=1}^n (x-e^{2\pi i c_i})$.

Recall that $A_c$ is graded by the action of $\ad h$ and the zeroth component is $\CN[h]$. We have the following proposition, similar to Proposition 3.1 from~\cite{EKRS}.
\begin{prop}
Let $C$ be a good contour. Then any $g_t$-twisted trace $T$ on $A_c$ is zero on $\ad z$ eigenspaces of nonzero weight and given on $\CN[z]$ by \[T(R(z))=\int_C R(x) w(x) dx,\qquad R\in \CN[z],\]
where $w$ is a weight function defined by the formula $w(t)=e^{2\pi i c x}\frac{G(e^{2\pi i x})}{\bP(e^{2\pi i x})}$ and $G$ is a polynomial of degree at most $n-1$ such that $G(0)=0$ if $c=0$.

\end{prop}
\begin{proof}
We note that $w(x+1)=tw(x)$. Conditions on $G$ imply that $w$ decays exponentially when $|\Im z|$ goes to infinity. Since $C$ is good, the integral $\int_C R(z) w(z) dz$ is defined for all $R\in \CN[z]$.

Proposition 2.3 from~\cite{EKRS} says that $T$ is a trace if and only if $T$ is supported on $\CN[h]$ and $T\big(S(z-\tfrac12)P(z-\tfrac12)\big)=tT(S(z+\tfrac12)P(z+\tfrac12)$. Similarly to the proof of Proposition 3.1 from~\cite{EKRS} we have \[T\big(S(z-\tfrac12)P(z-\tfrac12)-tS(z+\tfrac12)P(z+\tfrac12)\big)=-\int_{\partial U}S(x-\tfrac12)P(x-\tfrac12)w(x)dx,\] where $U$ is the region between $C$ and $C+1$, so that the boundary of $U$ is $C+1$ in positive direction and $C$ in negative. It is enough to prove that $P(x-\tfrac12)w(x)$ has no poles between $C$ and $C+1$. By definition of $C$ the poles of $w$ between $C$ and $C+1$ are contained in $\{c_1,\ldots,c_n\}$. The roots of $P$ are $c_1-\tfrac 12,\ldots,c_n-\tfrac 12$, so the roots of $P(z-\tfrac12)$ are $c_1,\ldots,c_n$. It follows that $P(x-\tfrac12)w(x)$ has no poles between $C$ and $C+1$.

We obtained the subspace of $g_t$-twisted traces of dimension $n$ if $t\neq 1$ and $n-1$ if $t=1$. This is exactly the dimension of the space of traces from Corollary 2.4 in~\cite{EKRS}.
\end{proof}

From now on, we will not need $t$ in our computations, only $c$. Since we use $x$ to express elements of $M_{c,c'}$, we will use $t$ as an integration variable.

Now we start computing the cone of positive definite Hermitian forms for fixed $\rho$. Note that different $\ad z$ eigenspaces are orthogonal with respect to $(\cdot,\cdot)$, hence it is enough to check the condition $(m,m)>0$ for $m$ in some eigenspace of $\ad z$. Suppose that $(\cdot,\cdot)$ is positive definite. When $m=x^j R_j(z)R(z)$, we have 
\begin{multline*}
(m,m)=T(m\phi(m))=T(x^j R_j(z) R(z) x^{-j} S_{-j}(z) \ovl{R}(-z))=\\
T(R_j(z-j)R(z-j)S_{-j}(z)\ovl{R}(-z))=T(R_j(z-j)S_{-j}(z)R(z-j)\ovl{R}(-z))=\\
\int_C R(t-j)\ovl{R}(-t) R_j(t-j)S_{-j}(t)w(t)dt.
\end{multline*}

We use the same strategy as in~\cite{EKRS}: we try to shift the contour $C$ to $i\RN+\tfrac{j}2$. If there are poles between $C$ and $i\RN+\tfrac{j}2$, we prove that this integral is negative for some $R$, a contradiction. If there are no poles, we use that polynomials are dense in $L^2(\RN,\omega)$ for an exponentially decaying weight $\omega$ to get that $R_j(z-j)S_{-j}(z)w(z)$ should be positive on $i\RN+\tfrac{j}2$. This gives a condition on $G$ similar to the one in~\cite{EKRS}.

We say that index $i$ is bad if there exists $j$ such that the function \[f(x)=\frac{R_j(x-j)S_{-j}(x)}{(e^{2\pi i x}-e^{2\pi i c_i})}\] has poles in the closed region between $C$ and $\frac{j}{2}+i\RN$. If this holds, we say that $j$ is bad for $i$, in the other case we say that $j$ is good for $i$.

\begin{lem}
\label{LemDescriptionOfBadI}
An index $i$ is bad if and only if $\Re c_i+\Re c'_i\le 0$ or \linebreak $\Re c_i+\Re c'_i\geqslant 2$.
\end{lem}
\begin{proof}
Recall that the intersection of roots of $R_j(z)$ with $c_i+\ZN$ is \linebreak $\{c'_i,c'_i+1,\ldots,c_i-j-1\}$ in the case when $c_i-c'_i-j>0$ and empty otherwise. For $R_j(z-j)$ this becomes $c'_i+j,\ldots,c_i-1$ or empty. For $S_j$ this intersection equals to $\{c_i,c_i+1,\ldots,c'_i+j-1\}$ in the case when $c'_i-c_i+j>0$ and empty otherwise.

Suppose that $i$ is bad, take $j$ that is bad for this $i$. Denote by $L$ the line $\frac{j}{2}+i\RN$. Assume that $f$ has a pole to the left of $C$ and to the right of $L$. Recall that the set $c_i-\ZN_{>0}$ is to the left of $C$ and $c_i+\ZN_{\geqslant 0}$ is to the right of $C$. It follows that this pole of $f$ is $c_i-k$ for some $k>0$. This pole is to the right of $L$, hence $\Re c_i-k\geqslant \frac{j}{2}$. All roots of the denominator $e^{2\pi i z}-e^{2\pi i c_i}$ of $f$ are simple, hence $c_i-k$ cannot be a root of $R_j(z-j)S_{-j}(z)$.

If $c'_i+j>c_i-k$ then $c_i-k$ cannot be a root of $R_j(z-j)S_{-j}(z)$. If $c'_i+j\le c_i-k$ then $c_i-c'_i-j\geqslant k>0$, so that $R_j(z-j)$ has roots $c'_i+j,\ldots,c_i-1$, hence it has root $c_i-k$.

So the condition we get is $c'_i+j\ge c_i-k+1$. Hence there are two conditions on $j$: \[\frac{j}{2}\le \Re c_i-k,\]
\[j\ge c_i-c'_i-k+1.\] They can be satisfied by an integer $j$ if an only if \[c_i-c'_i-k+1\le \lfloor 2\Re c_i\rfloor-2k.\]

There exists such integer $k>0$ if and only if $k=1$ works:
\[c_i-c'_i\le \lfloor 2\Re c_i\rfloor -2.\]

The left-hand side is an integer, so we don't need to take floor function on the right. We also have $c_i-c'_i=\Re c_i-\Re c'_i$. In the end we get $\Re c_i+\Re c'_i\ge 2$.

If there exists a pole of $f$ to the left of $L$ and to the right of $C$, the reasoning is similar: the poles is $c_i+k$ for some $k\geq 0$ and $\Re c_i+k\leqslant \frac{j}{2}$. It works if $c_i+k$ is not a root of $S_{-j}(z)$, this is equivalent to $c'_i+j-1<c_i+k$. So the two conditions are 
\[j\ge 2\Re c_i+2k,\]
\[j\le c_i-c'_i+k.\]

Similarly to the above they can be satisfied when \[  2\Re c_i+2k\le c_i-c'_i+k.\] It is enough to take $k=0$ here. In the end we get $\Re c_i+\Re c'_i\le 0$.
\end{proof}

We claim that the cone of positive traces on $M_{c,c'}$ is isomorphic to the cone from~\cite{EKRS}, where instead of bad and good roots we count bad and good indices $i$.

\begin{prop}
\label{PropBadPolesNoForm}
Suppose that there exists a bad $i$ such that $w$ has a pole in $c_i$. Then $w$ does not give a positive definite form on $M_{c,c'}$.
\end{prop}
\begin{proof}
The proof of this proposition is very similar to subsections 4.3-4.4 of~\cite{EKRS} with simplifications because $w$ has only simple poles.

Assume that $w$ gives a positive definite form.

We fix bad $i$ and $j$ that is bad for this $i$. We have \begin{multline*}(x^j R_j(h)R(h),x^j R_j(h)R(h))=\\
T(R_j(z-j)S_{-j}(z)R(z-j)\ovl{R}(-z))=\\
\intl_C R_j(s-j)S_{-j}(s)R(s-j)\ovl{R}(-s)w(s)ds
\end{multline*}

Let $S$ be a polynomial such that $S(t)w(t)$ has no poles between $C$ and $i\RN+\frac{j}{2}$. Then for any $R\in \CN[x]$ we have 
\begin{multline*}
(x^j R_j(z)R(z)S(z+j),x^j R_j(z)R(z)S(z+j))=\\
\intl_C R_j(t-j)S_{-j}(t)R(t-j)S(t)\ovl{R}(-t)\ovl{S}(j-t)dt=\\
\intl_{i\RN+\frac{j}{2}} R_j(t-j)S_{-j}(t)R(t-j)S(t)\ovl{R}(-t)\ovl{S}(j-t)dt
\end{multline*}
We will use Lemma 4.2 from~\cite{EKRS}:
\begin{lem}
\label{LemClosureOfPolynomials}
\begin{enumerate}
Suppose that $w(x)\ge 0$ is a measurable function on the real line such that $w(x)<c e^{-b|x|}$ for some $c,b>0$, $1\le p<\infty$.
\item
Suppose that $H$ is a continuous complex-valued function on $\RN$ with finitely many zeroes and at most polynomial growth at infinity. Then the set $\{H(x)S(x)\mid S(x)\in\CN[x]\}$ is dense in the space $L^p(\RN,w)$.
\item
Suppose that $M(x)$ is a nonzero polynomial nonnegative on the real line. Then the closure of the set $\{M(x)S(x)\ovl{S}(x)\mid S(x)\in \CN[x]\}$ in $L^p(\RN,w)$ is the subset of almost everywhere nonnegative functions.
\end{enumerate}
\end{lem}

We have \[\intl_{i\RN+\frac{j}{2}} R_j(t-j)S_{-j}(t)R(t-j)S(t)\ovl{R}(-t)\ovl{S}(j-t)w(t)dt>0\] for all $R\in \CN[x]$. Multiplying $w$ by $\pm i$, we can change $dt$ to $\abs{dt}$, a positive measure. Using Lemma~\ref{LemClosureOfPolynomials}(2) for \[w=R_j(t-j)S_{-j}(t)w(t)\] and \[M=S(t)\ovl{S}(j-t)\] after the change of argument $t\mapsto it+\frac{j}{2}$ we deduce that \[R_j(t-j)S_{-j}(t)w(t)\] is nonnegative on the line $i\RN+\frac{j}{2}$.

In particular, $R_j(t-j)S_{-j}(t)w(t)$ has poles of even order on the line $i\RN+\frac{j}{2}$. On the other hand all poles of $w$ are simple. Therefore $R_j(t-j)S_{-j}(t)w(t)$ has no poles on the line $i\RN+\frac{j}{2}$.

Since $i$ is bad, we deduce that $R_j(t-j)S_{-j}(t)w(t)$ has poles strictly between $C$ and $i\RN+\frac{j}{2}$. We write \begin{multline*}T(R(z)R_j(z-j)S_{-j}(z))=\intl_C R(t)R_j(t-j)S_{-j}(t)w(t)dt=\\
\intl_{i\RN+\frac{j}{2}} R(t)R_j(t-j)S_{-j}(t)w(t)dt+\Phi(R),
\end{multline*}
where $\Phi(R)$ is a nonzero linear functional of the form $\sum a_i R(t_i)$, $a_i\in \CN$, $t_i$ are poles of $R_j(t-j)S_{-j}(t)w(t)$ between $C$ and $i\RN+\frac{j}{2}$. We get a contradiction with the following lemma for $w=R_j(t-j)S_{-j}(t)w(t)$ after a change of argument $t\mapsto it+\frac{j}{2}$.
\begin{lem}
Suppose that $w(t)$ is almost everywhere nonnegative function on the real line such that $w(t)<be^{-c\abs{t}}$ for some $b,c>0$, $\Phi$ is a nonzero linear functional on $\CN[t]$ of the form \[\Phi(R)=\sum_{i=1}^l a_i R(t_i),\] where $a_i\in \CN$, $t_i\notin \RN$,  \[T(R)=\intl_{\RN} w(t) R(t) dt+\Phi(R).\] Then there exists $R\in \CN[t]$ such that $T(R(t)\ovl{R}(t))\notin\RN_{\geq 0}$.
\end{lem}
\begin{proof}
Let $S$ be a polynomial such that $S=\ovl{S}$ and $S(t_1)=\cdots=S(t_l)=0$. It follows that  \begin{equation}
\label{EqPhiSPEqualsZero}
\Phi(SP)=\Phi(\ovl{S}P)=\{0\}
\end{equation} for any polynomial $P$. Let $R$ be any polynomial. Using Lemma~\ref{LemClosureOfPolynomials}(1) for $H=S$ we find a sequence of polynomials $M_n$ such that $SM_n$ tends to $R$ in $L^2(\RN,w)$. Using~\eqref{EqPhiSPEqualsZero} we have 
\[T((R-SM_n)(\ovl{R}-\ovl{S}\ovl{M_n}))=\|R-SM_n\|^2_{L^2(\RN,w)}+\Phi(R\ovl{R}).\]
Since $\|R-SM_n\|_{L^2(\RN,w)}$ tends to zero, it is enough to find $R\in\CN[t]$ such that $\Phi(R\ovl{R})$ is not a nonnegative real number.

Let $\Phi(R)=\sum_{i=1}^k a_i R(t_i)$, where $a_1\neq 0$. Taking $a_2=0$ if necessary, we can assume that $t_2=\ovl{t_1}$. Let $p,q$ be complex numbers. Let $R$ be a polynomial such that $R(t_1)=p$, $R(t_2)=q$, $R(t_i)=0$ for $i>2$. Then $\Phi(R\ovl{R})=a_1p\ovl{q}+a_2q\ovl{p}$. We can find $p,q$ such that $a_1p\ovl{q}+a_2q\ovl{p}$ is not a nonnegative real number. The lemma follows.
\end{proof}
\end{proof}

Now we assume that $w$ does not have poles at $c_i$ for all bad $i$. In this case we can write $w(x)=e^{2\pi i cx}\frac{G(e^{2\pi ix})}{\bP(e^{2\pi ix})}$, where the new $\bP$ has roots at $e^{2\pi i c_i}$ for all good $i$. We have the following
\begin{prop}
\label{PropPositiveFormMeansPositiveWeight}
The form $(m,n)=T(m\phi(n))$ is positive definite if and only if $R_j(t-j)S_{-j}(t)w(t)\geq 0$ for all $j$ and $t\in i\RN+\tfrac{j}2$.
\end{prop}
\begin{proof}
Recall that \[(x^j R_j(z)R(z),x^j R_j(z)R(z))=\int_C R_j(t-j)S_{-j}(t)R(t-j)\ovl{R}(-t)w(t)dt.\]
Since $w$ is good, we can take $i\RN+\tfrac{j}{2}$ instead of $C$ in this integral. We can also change $dt$ to $\abs{dt}$ for convenience. Hence $(\cdot,\cdot)$ is positive definite if and only if
\[\int_{i\RN+\tfrac{j}2}R_j(t-j)S_{-j}(t)R(t-j)\ovl{R}(-t)w(t)\abs{dt}>0\] for all integer $j$ and nonzero polynomials $R$. Using Lemma~\ref{LemClosureOfPolynomials}  after the change of argument $t\mapsto it+\tfrac{j}2$ with $M=1$, we see that $R_j(t-j)S_{-j}(t)w(t)$ should be nonnegative on $i\RN+\tfrac{j}2$.
\end{proof}

It remains to understand  $R_j(t-j)S_{-j}(t)w(t)$ when $\Re t=\tfrac{j}2$. We start with describing the behavior of $R_j(t-j)S_{-j}(t)$ on $i\RN+\tfrac{j}2$. Recall that $\phi\colon M_{c,c',\rho^{-1}}\to M_{c',c}$ is defined up to a constant. Since $w$ can also be multiplied by any constant, we can choose any $\phi$ we like and the answer will be the same.

\begin{lem}
\label{LemHowRSBehaves}
The set of roots of $R_j(t-j)S_{-j}(t)$ is $\{c_i,\ldots,c'_i+j-1\}$, $\{c'_i+j,\ldots,c_i-1\}$ or empty depending on the sign of $c'_i-c_i+j$. We can choose $\phi$ such that for all $j$ the polynomial $(ai^n)^{-j} R_j(t-j) S_{-j}(t)$ is real on the line $\Re t=\frac{j}{2}$ and positive when $\Re t=\frac{j}{2}$ and $\Im t$ is large enough.
\end{lem}
\begin{proof}
Recall that the intersection of $c_i+\ZN$ with the roots of $R_j(t)$ equals to $\{c'_i,\ldots,c_i-j-1\}$, it has size $\max(c_i-c'_i-j,0)$. Lemma~\ref{LemCCprimeRhoIsomorphic} says that the intersection of $c_i+\ZN$ with the set of roots of $S_{-j}$ equals to $\{c_i,\ldots,c'_i+j-1\}$, and has size $\max(c'_i-c_i+j,0)$. It also says that the leading coefficient of $S_{-j}$ equals to $C_{\phi}(-1)^{\deg R_j}a^j$, where $C_{\phi}$ is a constant that depends only on $\phi$.

Hence one of the intersections of $c_i+\ZN$ with roots is empty and the other has size $\abs{c_i-c'_i-j}$ and may be also empty.

Recall that there exists index $k$, possibly equal to $i$, such that $c_i+\ovl{c'_k}=1$, $c_k+\ovl{c'_i}=1$. It follows that $c_i-c'_i=c_k-c'_k$. Hence $c_i-c'_i-j=c_k-c'_k-j$, in particular they have the same sign.

Suppose that the intersection of $c_i+\ZN$ with the set of roots of \linebreak $R_j(t-j)S_{-j}(t)$ is $\{c'_i+j,\ldots,c_i-1\}$. In this case the intersection of $c_k+\ZN$ with the set of roots of $R_j(t-j)S_{-j}(t)$ is $\{c'_k+j,\ldots,c_k-1\}$. We have $-\ovl{c'_i+j}=j-\ovl{c'_i}=j+c_k-1$. In this case we see that the intersection of $\{c_i,c_k\}+\ZN$ with the set of roots of $R_j(t-j)S_{-j}(t)$ is symmetric with respect to the line $\Re t=\frac{j}{2}$. The other case is done similarly.

Hence the roots of $R_j(t-j)S_{-j}(t)$ are symmetric with respect to the line $\Re t=\frac{j}2$.

Note that $R_j$ is monic and the leading coefficient of $S_{-j}$ is $C_{\phi}(-1)^{\deg R_j}a^j$. It follows that the argument of $R_j(t-j)S_{-j}(t)$ tends to the argument of \[i^{\deg R_j+\deg S_{-j}}C_{\phi}(-1)^{\deg R_j}a^j=C_{\phi}i^{\deg S_{-j}-\deg R_j}a^j\] when $t$ tends to $i\infty$.

We can compute $\deg S_{-j}-\deg R_j$ as the number of roots of $S_{-j}$ minus the number of roots of $R_j$. We see from the description of roots above that each $i=1,\ldots,n$ contributes $j-c'_i-c_i$ to this expression, so that $\deg S_{-j}-\deg R_j=nj-\sum c'_i-\sum c_i$. Hence the argument of $R_J(t-j)S_{-j}(t)$ tends to the argument
\[(ai^n)^j C'_{\phi}\]when $t$ tends to $i\infty$, where $C'_{\phi}=i^{-\sum c'_i-\sum c_i}C_{\phi}$. Choosing $C'_{\phi}=1$ we get that the argument of $(ai^n)^{-j}R_j(t-j)S_{-j}(t)$ tends to zero when $t$ tends to $i\infty$. On the other hand, this polynomial does not change argument when $\Im t$ is large enough and $\Re t=\frac{j}{2}$, hence it is positive. Since the roots of $R_j(t-j)S_{-j}(t)$ are symmetric with respect to $\Re t=\frac{j}{2}$, the polynomial $(ai^n)^{-j}R_j(t-j)S_{-j}(t)$ is real on this line.
\end{proof}

\begin{prop}
\label{PropWGood}
In this case $w$ gives a positive definite form on $M_{c,c'}$ if and only if $G$ has certain behavior on the real line: $G(x)$ is nonnegative for $x\in \RN$ in the case when $\rho=\rho_+$, $G(x)$ is nonnegative for $x>0$ and nonpositive for $x<0$ in the case when $\rho=\rho_-$.
\end{prop}
\begin{proof}
The condition on $w$ implies that \[\big(x^j R_j(z)R(z),x^jR_j(z)R(z)\big)=\intl_{i\RN+\frac{j}{2}} R_j(t-j)S_{-j}(t)R(t-j)\ovl{R}(-t)w(t)\abs{dt}.\] Polynomial $R(t-j)\ovl{R}(-t)$ is nonnegative on $i\RN+\frac{j}{2}$. It remains to check that $R_j(t-j)S_{-j}(t)w(t)$ is nonnegative on $i\RN+\frac{j}{2}$.  Lemma~\ref{LemHowRSBehaves} says that $(ai^n)^{-j}R_j(t)S_{-j}(t)$ is positive when $\Re t=\frac{j}{2}$ and $\Im t$ is large enough. 

We have $a=\eps e^{-\pi i c}i^n$, where $\rho=\rho_{\eps}$. It follows that \[\eps^j e^{\pi i c j} (-1)^{nj} R_j(t-j)S_{-j}(t)\] is positive when $\Re t=\frac{j}{2}$ and $\Im t$ is large enough.

We note that the zeroes of $\bP(e^{2\pi i t})$ on $i\RN+\frac{j}{2}$ are simple and in one-to-one correspondence with the roots of $R_j(t-j)S_{-j}(t)$ on $i\RN+\frac{j}{2}$, this follows from the definition of a good index. Hence $w(t)R(t-j)S_{-j}(t)$ does not change argument on $i\RN+\frac{j}{2}$ if and only if $G$ does not have roots on $(-1)^j \RN_{>0}$.

Hence the necessary condition for positivity is that $G$ does not have roots on $\RN\setminus\{0\}$. If this condition holds then $w(t)R(t-j)S_{-j}(t)$ does not change argument on $i\RN+\frac{j}{2}$ for all $j$.

We have \[w(t)=e^{2\pi i c t}\frac{G(e^{2\pi it})}{\bP(e^{2\pi it})}.\] When $\Re t=\frac{j}{2}$, the function $e^{2\pi i c t}$ has argument $\pi c j$. 

It remains to look at the behavior of $\frac{G(e^{2\pi i t})}{\bP(e^{2\pi i t})}$ when $\Re t=\frac{j}{2}$ and $\Im t$ tends to infinity. Suppose that the lowest term in $G(x)$ is $sx^k$. Then $\frac{G(e^{2\pi i t})}{\bP(e^{2\pi i t})}$ has sign $\frac{(-1)^{kj} s}{\bP(0)}$. We get the condition $\eps^j (-1)^{nj+kj}\frac{s}{\bP(0)}$ should be positive for all $j$. This means that $\eps=(-1)^{n+k}$ and $\frac{s}{\bP(0)}$ is positive. The proposition follows.
\end{proof}
\begin{rem}
For $\rho=\rho_-$ the sign of our polynomials is flipped compared to~\cite{EKRS}. This happened because in our case we also have a choice of $\phi$: if we take $-\phi$ instead of $\phi$, we should take $-T$ instead of $T$.
\end{rem}
So we have proved the following
\begin{thr}
\label{ThrFilteredAnswer}
Let $m$ be the number of good indices $i$. Then the dimension of the cone of positive forms is the same as in~\cite{EKRS}, namely:
\begin{itemize}
\item $n-1$ for even $n$ and $n-2$ for odd $n$ if $\rho=\rho_-$;

\item $n-1$ for even $n$ and $n$ for odd $n$ if $c\ne 0$ and $\rho=\rho_+$;

\item $n-3$ for even $n$ and $n-2$ for odd $n$ if $c=0$ and $\rho=\rho_+$.
\end{itemize}
If the dimension is $\leq 0$, the cone is empty.
\end{thr}
\section{Unitarizability for $q$-deformations}
\label{SecQUnitarizability}
\subsection{Construction of Hermitian form on a bimodule}
Let $q$ be a positive number. Consider the algebra $A$ that is generated by $x,x^{-1},D, D^{-1}$ with relations $Dx=q^2xD$. $A$ acts on the algebra $\RN[\CN]=\oplus_{\Re s\in [0,1)} x^s \CN[x,x^{-1}]$ by $x.P(x)=xP(x)$, $DP(x)=P(q^2x)$. Since $q$ is positive, any complex power of $q$ is well-defined.

Let $c_1,\ldots,c_n$ be complex numbers such that none of $c_i-c_j$ belong to the lattice $\ZN+\frac{\pi i}{\ln q}\ZN$. We have the following lemma:
\begin{lem}
\label{LemQAlgOfDiffOp}
Consider the subalgebra $A_c\subset A$ of operators that preserve $x^{c_i}\CN[x]$ for all $i=1,\ldots,n$. It is generated by $u=x$, $Z=D$, $Z^{-1}$ and $v=x^{-1}(D-q^{2c_1})\cdots (D-q^{2c_n})$. The set of defining relations is \[ZuZ^{-1}=q^2u,\quad ZvZ^{-1}=q^{-2}v,\quad uv=P(q^{-1}Z), \quad vu=P(qz),\] where $P(z)=q^n (z-q^{2c_1-1})\cdots (z-q^{2c_n-1})$.
\end{lem}
\begin{proof}
We see that $A_c$ contains $u,v,Z, Z^{-1}$.

Let $a$ be an element of $A_c$, $a=\sum x^k P_k(D)$, where $P_k$ are Laurent polynomials. We have $a x^{c_i+l}=\suml_k x^{k+c_i+l} P_k(q^{2c_i+2l})$. Since $ax^{c_i+l}$ belongs to $x^{c_i}\CN[x]$, we see that $P_k(q^{2c_i+2l})=0$ for all $k,l$ such that $k+l<0$. In particular for $k<0$ the polynomial $P_k$ has at least $nk$ roots: $P_k(q^{2c_i+2l})=0$ for all $i=1,\ldots,n$ and $l=0,\ldots,-k-1$. By our assumption numbers $q^{2c_i+2l}$ are all distinct.

Fix a negative $k$. By induction on $k$ we prove that \[v^{-k}=x^k (D-q^{2c_1})(D-q^{2c_1+2})\cdots (D-q^{2c_1+2k-2})\cdots (D-q^{2c_n+2k-2}).\] Hence there exists a polynomial $Q_k$ such that $x^k P_k(D)=v^{-k}Q_k(D)$. For $k\geq 0$ we have $x^k P_k(D)=u^k P_k(D)$. Therefore $a$ belongs to the subalgebra of $A$ generated by $u,v,D$.

The elements $u,v,Z$ satisfy the relationships of the lemma statement. Using these relationships, we see that the elements $u^k Z^l$ and $v^k Z^l$ span $A_c$, where we take $k\geq 0$. Using the action of $A_c$ on $\RN[\CN]$, we see that $u^k Z^l$ and $v^k Z^l$ are a basis of $A_c$. The lemma follows.
\end{proof}
We will require that $n=2m$ be even. We multiply $v$ by $Z^{-m}$ from the right so that $P$ becomes $P(Z)=q^mZ^m+\cdots+tZ^{-m}$ for some nonzero complex $t$. More precisely, by Vieta's formula, $t=q^m\prod_{i=1}^n q^{\sum_{i=1}^{n}(2c_i-1)}$. Then we multiply $v$ by a complex number so that the coefficient of $P$ on $Z^m$ becomes $q^{m-\sum c_i}$. The coefficient of $P$ on $Z^{-m}$ becomes $q^{\sum c_i-m}$. We can write this as \[P(z)=Z^m q^{m-\sum c_i}Z^m+\ldots +Z^{-m}q^{\sum c_i-m}=\prod_{i=1}^n (\sqrt{z}q^{\tfrac{1}{2}-c_i}-\sqrt{z}^{-1}q^{c_i-\tfrac{1}{2}}).\] After doing this we define $P_c(z)=P(z)$. This will be convenient for our computations: a Laurent polynomial of the form $aZ^m+\cdots+a^{-1}Z^{-m}$ is defined by its nonzero roots up to a sign, and the product or quotient (when polynomial) of two such polynomials again has this form.

Let $c=(c_1,\ldots,c_n), c'=(c'_1,\ldots,c'_n)$ be parameters such that $c_i-c'_i$ are integers. Consider the subset $M_{c,c'}\subset A$ of operators that send $x^{c'_i}\CN[x]$ to $x^{c_i}\CN[x]$ for each $i=1,\ldots,n$. This is naturally an $A_c$-$A_{c'}$-bimodule. Moreover, we have a natural map from $M_{c,c'}\otimes_{A_{c'}}M_{c',c}$ to $A_c$.

The proof of the following lemma is the same as the proof of Lemma~\ref{LemMoritaEquivalence}:
\begin{lem}
\label{LemQMoritaEquivalence}
The maps $\phi\colon M_{c,c'}\otimes_{A_{c'}}M_{c',c}\to A_c$ and $\psi\colon M_{c',c}\otimes_{A_c}M_{c,c'}\to A_{c'}$ give a Morita equivalence between $A_c$ and $A_{c'}$.
\end{lem}
\begin{lem}
\label{LemDescriptionOfRj}
We have $M_{c,c'}=\plusl_{j\in\ZN} x^j R_j(Z)\CN[Z,Z^{-1}]$, where $R_j(Z)$ is a monic polynomial with the following set of roots: we start with an empty set and for all $i=1,\ldots,n$ such that $c'_i\le c_i-j-1$ we add $q^{2c'_i},q^{2c'_i+2},\ldots,q^{2c_i-2j-2}$.
\end{lem}
\begin{proof}
Let $m$ be an element of $M=M_{c,c'}$, $m=\sum x^k P_k(D)$, where $P_k$ are Laurent polynomials. For all triples $i,k$ and $l\geq 0$ such that $c'_i+k+l< c_i$, we have $P_k(q^{2c'_i+2l})=0$. For fixed $i,k$ this is equivalent to $0\le l\le c_i-c'_i-1-k$. This gives $P_k(q^{2c'_i})=\cdots=P_k(q^{2c_i-2k-2})=0$. The lemma follows.
\end{proof}

Below we will write $q^{2c'_i},\ldots,q^{2c_i-2j-2}$ to mean $c_i-j-c'_i$ numbers when $c_i-j>c'_i$ and empty set otherwise.


Now we assume that $A_{c'}$ is isomorphic to $\ovl{A_c}$ and both maps $\rho$ are given by $u\mapsto av$, $v\mapsto bu$, $Z\mapsto Z^{-1}$, the same formula as in~\cite{K}. 

This isomorphism exists when 
\begin{equation}
\label{EqRhoInQCase}
abP_c(qZ)=abvu=\rho(uv)=\rho(P_{c'}(q^{-1}Z))=\ovl{P_{c'}}(q^{-1}Z^{-1}).
\end{equation} From this we deduce the following: $P_c(z_0)=0$ if and only if $P_{c'}(\ovl{z_0^{-1}})=0$. Hence for any $i$ from $1$ to $n$ there exists $j$ from $1$ to $n$ such that $q^{2c_i-1}=\ovl{q^{1-2c_j'}}$. It follows that \[q^{2c_i+\ovl{2c_j'}-2}=e^{2\ln q(c_i+\ovl{c'_j}-1)}=1,\] hence $c_i+\ovl{c'_j}-1$ is a multiple of $\frac{\pi i}{\ln q}$. Shifting $c_i$ by $\frac{\pi i}{\ln q}$ does not change $2c_i$,  so we can assume that $c_i+\ovl{c'_j}-1=0$. We can do this for all $i=1,\ldots,n$. Similarly to the previous section from $c_j+\ovl{c'_k}-1=0$ we deduce that $k=i$. Hence there exists an involution $\sigma$ such that \begin{equation}
\label{EqRealityCondition}
c_i+\ovl{c'_{\sigma(i)}}=1
\end{equation} for all $i$.

Note that for any complex number $s$ with $\abs{s}=1$ we have $\rho(sZ)=\ovl{s}\rho(Z)=s^{-1}Z^{-1}=(sZ)^{-1}$, hence we can change $Z$ to $sZ$ in both algebras $A_c$, $A_{c'}$.

Taking $a=q^{2i\lambda}$ shifts all $c_j$ by $-i\lambda$. Hence we can choose generator $Z$ such that $\sum c_i$ is real. Note that \begin{multline}
\label{EqSumCiIsInteger}
2\sum c_i=2\sum\Re c_i=\sum \Re c_i+\Re c_{\sigma(i)}=\\
\sum \Re (c_i-c_i')+\Re c'_i+\Re c_{\sigma(i)}=\sum(c_i-c_i'+1)
\end{multline} is an integer. We require that $\sum c_i$ is an integer, not half-integer, so that we don't obtain square root  of $Z$ in the proof of Lemma~\ref{LemQMIsomorphicToMRho} below. This means that the number of singletons $i=\sigma(i)$ with $\Re c_i$ half-integer is even. 

If we interchange $c$ and $c'$ in~\eqref{EqSumCiIsInteger} and add~\eqref{EqSumCiIsInteger} we get $2\sum c_i+2\sum c'_i=2n$, hence $\sum c_i+\sum c'_i=n$. Recall that $P_c(Z)=q^{m-\sum c_i}Z^m+\cdots+q^{\sum c_i-m}Z^{-m}$, $P_{c'}(Z)=q^{m-\sum c'_i}Z^n+\cdots+q^{m-\sum c'_i}$, where $n=2m$. Note that all coefficients of $Z^m$ and $Z^{-m}$ are real and $m-\sum c_i=\sum c'_i-m$. It follows that $P_c(Z)$ and $\ovl{P_{c'}}(Z^{-1})$ have the same leading coefficient. Comparing this with~\eqref{EqRhoInQCase} we deduce that $ab=1$. We want $u\in A_c,A_{c'}$ to be exactly $x$, so we will leave it at that and allow any $a,b$ such that $ab=1$. Then $a=\abs{a}e^{2\pi i s}$, $b=\abs{a}^{-1}e^{-2\pi i s}$, where $s\in [0,1)$. It follows from~\eqref{EqRhoInQCase} that $P_c(z)=\ovl{P_{c'}}(z^{-1})$.

On the other hand, any $a,b$ with $ab=1$ and parameters $c,c'$ such that for some involution $\sigma$ we have $c_i+\ovl{c'_{\sigma(i)}}=1$ will give an antilinear isomorphism $\rho$ between $A_c$ and $A_{c'}$ as above.

Recall that we have the notion of $\rho$-invariant sesquilinear form and Proposition~\ref{PropFormsAndTraces} says that sesquilinear $\rho$-invariant forms on $M$ are in one-to-one correspondence with $\rho^2$-twisted traces on $M\otimes_{A_{c'}} M_{\rho^{-1}}$. We use the same strategy as in Section~\ref{SecUsualUnitarizability} in order to describe $\rho$-invariant forms: prove that $M_{c,c',\rho^{-1}}$ is isomorphic to $M_{c',c}$ to get $\rho^2$-twisted traces on $M_{c,c'}\otimes_{A_{c'}}M_{c',c}\cong A_c$.

In order to check that two usual polynomials are equal to each other, it is enough to check that they have the same roots and the same leading coefficient. For Laurent polynomials we should also check that they have root or pole of the same order at zero. In order to deal with this, we will consider {\it balanced} Laurent polynomials, meaning they have the form $ax^N+\cdots+bx^{-N}$. It is possible for $R_j$ to have odd number of nonzero roots, so we allow $N$ to be a half-integer. We will construct an isomorphism between $M_{c,c'}[\sqrt{Z}]$ and $M_{c',c,\rho}[\sqrt{Z}]$ below and check that it gives an isomorphism between $M_{c,c'}$ and $M_{c',c,\rho}$.

From now on we shift all $R_j$ by a power of $\sqrt{Z}$ so that they become symmetric.

\begin{lem}
\label{LemQMIsomorphicToMRho}
\begin{enumerate}
\item
We have $M_{c,c'}\cong M_{c',c,\rho}$. Denote this isomorphism by $\phi$. We also get that $\phi$ is an isomorphism from $M_{c,c',\rho^{-1}}$ to $M_{c',c}$.
\item
  Let us write $\phi$ as $\phi(x^{j}R_j(Z)R(Z))=x^{-j}S_{-j}(Z)\ovl{R}(Z^{-1})$. Then $\frac{a^2S_{j+2}(q^2z)}{S_j(z)}$ is positive for all $z$ with $\abs{z}=1$.
\item
We can choose $\phi$ so that the polynomials $a^{-j}R_j(q^{-j}z)S_{-j}(q^j z)$ are real for all $j$.
\end{enumerate}

\end{lem}
\begin{proof}
The second statement follows from the first after twisting the action from both sides by $\rho^{-1}$.

We will use $.$ to denote the action of $A_c,A_{c'}$ on $M_{c',c,\rho}$, so that $Z.m=Z^{-1}m, u.m=avm$ and so on.


We will construct the map $\phi$ similarly to the proof of Lemma~\ref{LemCCprimeRhoIsomorphic}. Let $\phi(x^jR_j(Z)R(Z))=x^{-j}S_{-j}(Z)\ovl{R}(Z^{-1})$, where $S_{-j}$ has symmetric degree and has roots similar to the roots of $R_j$ but is not necessarily monic.

Note that $\phi$ is antilinear and satisfies $\phi(mZ)=\phi(m)Z^{-1}$. We also have $\phi(ZmZ^{-1})=Z^{-1}\phi(m)Z$. It follows that $\phi(Zm)=Z^{-1}\phi(m)$.

Since elements $x^j R_j(Z)$ form a $\CN[Z,Z^{-1}]$ basis of $M_{c,c'}$ both for the left and for the right action of $\CN[Z,Z^{-1}]$, it is enough to prove $\phi(um)=u.\phi(m)$ for $m=x^j R_j(Z)$ and similarly for the other conditions.

We want to prove that $\phi(um)=u.\phi(m)=avm$ for $m=x^j R_j(Z)$. We have $um=xm=x^{j+1}R_j(Z)$. Denote $\frac{R_j(Z)}{R_{j+1}(Z)}$ by $L_j(Z)$. We get $\phi(um)=x^{-j-1}S_{-j-1}(Z)\ovl{L_j}(Z^{-1})$.

By definition \begin{multline*}
 av\phi(m)=ax^{-1}P_{c'}(q^{-1}Z)\phi(m)=ax^{-1}P_{c'}(q^{-1}Z)x^{-j}S_{-j}(Z)=\\
 ax^{-j-1}P_{c'}(q^{-1-2j}Z)S_{-j}(Z).
 \end{multline*}

Denote $\frac{S_{-j-1}(Z)}{S_{-j}(Z)}$ by $M_{-j-1}(Z)$. It follows that \[\phi(um)=x^{-j-1}M_{-j-1}(Z)\ovl{L_j}(Z^{-1})S_{-j}(Z).\] 

So it is enough to prove that 
\begin{equation}
\label{EqLMEqualsPQ}
aP_{c'}(q^{-1-2j}Z)=M_{-j-1}(Z)\ovl{L_j}(Z^{-1}).
\end{equation} Note that $P_{c'}, M_{-j-1}$ and $L_j$ are all balanced. Hence it is enough to check that both sides have the same roots and choose $S_j$ such that both sides have the same leading coefficient.

Similarly to Lemma~\ref{LemCCprimeRhoIsomorphic} we see that $L_j(Z)$ has roots $q^{2c_i-2j-2}$ for all $i$ satisfying $c_i-j-1\geq c_i'$. Hence $\ovl{L_j}(Z^{-1})$ has roots $q^{2+2j-2\ovl{c_i}}=q^{2j+2c_\sigma(i)'}$, where $\sigma$ is the permutation such that $\ovl{c_i}+c'_{\sigma(i)}=1$ for all $i$. Similarly to the proof of Lemma~\ref{LemCCprimeRhoIsomorphic} we see that roots of $L_j$ come in pairs $q^{2c_i-2j-2}$, $q^{2c_{\sigma(i)}-2j-2}$.

Similarly to the above and to the proof of Lemma~\ref{LemCCprimeRhoIsomorphic}, we see that $M_{-j-1}(Z)$ is balanced and has roots $q^{2c'_i+2j}$ for all $i$ satisfying $c'_i+j\geq c_i$, so that $\ovl{L_j}(Z^{-1})M_{-j-1}(Z)$ has roots $q^{2c'_1+2j},\ldots,q^{2c'_n+2j}$. Polynomial $P_{c'}(q^{-1-2j}Z)$ also has roots $q^{2c'_1+2j},\ldots,q^{2c'_n+2j}$. It follows that $P_{c'}(q^{-1-2j}Z)$ and $\ovl{L_j}(Z^{-1})M_{-j-1}(Z)$ have the same set of roots. Multiplying each $S_k$ by its own nonzero constant we can make $aP_{c'}(q^{-1-2j}Z)$ equal to $\ovl{L_j}(Z^{-1})M_{-j-1}(Z)$ for all $j$.

Hence $\phi$ satisfies $\phi(um)=u.\phi(m)$.


It remains to check that $\phi(mu)=\phi(m).u=a\phi(m)v$:  we claim that $\phi(vm)=v.\phi(m)$ and $\phi(mv)=\phi(m).v$ follow from all the other conditions. Indeed, since $M_{c',c}$ is a torsion-free $A_{c'}$-module, it is enough to check that $u.\phi(vm)=uv.\phi(m)$. We have \[u.\phi(vm)=\phi(uvm)=\phi(P_c(Z)m)=P_c(Z).\phi(m)=uv.\phi(m).\] Similarly $\phi(mv)=\phi(m).v$ follows from $\phi(mv).u=\phi(mvu)$.

Hence it remains to prove that $\phi(mu)=a\phi(m)v$. As before we can check this for $m=x^j R_j(Z)$. We have \[mu=x^j R_j(Z)x=x^{j+1}R_j(q^2Z)=x^{j+1}R_{j+1}(Z)A_j(Z).\] Here, similarly to Lemma~\ref{LemCCprimeRhoIsomorphic} $A_j(Z)=\frac{R_{j}(q^2Z)}{R_{j+1}(Z)}$ is a polynomial that has roots $q^{2c'_i-2}$ for all $i$ such that $c'_i\le c_i-j-1$. 

We have $\phi(mu)=x^{-j-1}S_{-j-1}(Z)\ovl{A_j}(Z^{-1})$. Here $\ovl{A_j}(Z^{-1})$ is a polynomial in $Z^{-1}$ that has roots $Z=\ovl{(q^{2c'_i-2})^{-1}}=q^{2-2\ovl{c'_i}}=q^{2c_{\sigma(i)}}$.

We have \begin{multline*}
a\phi(m)v=ax^{-j} S_{-j}(Z)x^{-1}P_c(q^{-1}Z)=ax^{-j-1}S_{-j}(q^{-2}Z)P_c(q^{-1}Z)=\\
ax^{-j-1}S_{-j-1}(Z)\frac{P_c(q^{-1}Z)}{B_{-j-1}(Z)},
\end{multline*} where $B_{-j-1}(Z)=\frac{S_{-j-1}(Z)}{S_{-j}(q^{-2}Z)}$ has roots $q^{2c_i}$ for all $i$ such that $c_i\le c'_i+j$. Hence we should prove that
\begin{equation}
\label{EqABequalsPQ}
\ovl{A_j}(Z^{-1})B_{-j-1}(Z)=aP_c(q^{-1}Z).
\end{equation}

Similarly to the proof of Lemma~\ref{LemCCprimeRhoIsomorphic} and to the reasoning above we deduce that $\ovl{A_j}(Z^{-1})B_{-j-1}(Z)$ and $aP_c(q^{-1}Z)$ have the same set of roots. Since $A_j, B_{-j-1}$ and $P_c$ are all symmetric, it is enough to check that they have the same leading coefficient. In other words, \[\lim_{Z\to \infty}\frac{\ovl{A_j}(Z^{-1})B_{-j-1}(Z)}{aP_c(q^{-1}Z)}\] should be equal to one. We compute this limit as a product
\[\lim_{Z\to 0}\frac{\ovl{A_j}(Z)}{\ovl{L_j}(Z)}\lim_{Z\to\infty}\frac{B_{-j-1}(Z)}{M_{-j-1}(Z)}
\lim_{Z\to\infty}\frac{P_{c'}(q^{-1-2j}Z)}{P_c(q^{-1}Z)},\] since $\ovl{L_j}(Z^{-1})M_{-j-1}(Z)=aP_{c'}(q^{-1-2j}Z)$.

We have $L_j(Z)=\frac{R_j(Z)}{R_{j+1}(Z)}$, $A_j(z)=\frac{R_{j}(q^2Z)}{R_{j+1}(Z)}$. Hence \[\lim_{Z\to 0}\frac{\ovl{A_j}(Z)}{\ovl{L_j}(Z)}=q^{-\deg R_j}.\]

We have $M_{-j-1}(Z)=\frac{S_{-j-1}(Z)}{S_{-j}(Z)}$, $B_{-j-1}(Z)=\frac{S_{-j-1}(Z)}{S_{-j}(q^{-2}Z)}$, hence the leading coefficient of \[\lim_{Z\to\infty}\frac{B_{-j-1}(Z)}{M_{-j-1}(Z)}=q^{\deg S_{-j}}.\]

Recall that \[P_c(Z)=q^{-\sum c_i+m}Z^m+\cdots+q^{\sum c_i-m}Z^{-m},\] \[P_{c'}(Z)=q^{-\sum c'_i+m}Z^{-m}+\cdots+q^{\sum c'_i-m}Z^{-m},\] hence \[\lim_{Z\to\infty}\frac{P_{c'}(q^{-1-2j}Z)}{P_c(q^{-1}Z)}=q^{-2mj+\sum c_i-\sum c'_i}.\]

It follows that
\begin{equation}
\label{EqProductOfThreeLimits}
\lim_{Z\to 0}\frac{\ovl{A_j}(Z)}{\ovl{L_j}(Z)}\lim_{Z\to\infty}\frac{B_{-j-1}(Z)}{M_{-j-1}(Z)}
\lim_{Z\to\infty}\frac{P_{c'}(q^{-1-2j}Z)}{P_c(q^{-1}Z)}=q^{\deg S_{-j}-\deg R_j-2mj+\sum c_i-\sum c'_i}.
\end{equation}

Note that each $i$ adds exactly $c'_i+j-c_i$ to $\deg S_{-j}-\deg R_j$: either $S_{-j}$ has roots $c_i,\ldots,c'_i+j-1$ or $R_j$ has roots $c'_i,\ldots,c_i-j-1$. Hence $\deg S_{-j}-\deg R_j=\sum c'_i-\sum c_i+2mj$. It follows that the right-hand side in~\eqref{EqProductOfThreeLimits} equals to one. We deduce that $\ovl{A_j}(Z^{-1})B_{-j-1}(Z)$ and $aP_c(q^{-1}Z)$ have the same set of roots, hence they coincide.

It follows that $\phi$ is an isomorphism of $A_c-A_{c'}$ bimodules. It remains to check that it sends $M_{c,c'}\subset M_{c,c'}[\sqrt{Z}]$ to $M_{c',c,\rho}\subset M_{c',c,\rho}[\sqrt{Z}]$. For each $j$ the polynomial $R_j$ has either all integer monomials or all half-integer ones and similarly for $S_j$. So we have to check that either both $R_j$ and $S_{-j}$ have integer monomials or both have half-integer monomials. This depends on the number of roots: integer degrees for even number of roots, half-integer for odd. Hence $R_j$ and $S_{-j}$ should have the same parity of number of roots.

We will prove by downward induction on $j$ that $R_j$ and $S_{-j}$ have the same parity of number of roots. The base case is for $j$ large enough. In this case $R_j=1$ has no roots and $S_{-j}$ has roots $c_i,\ldots,c'_i+j-1$ for all $i$. Hence there are $\sum c'_i-\sum c_i+nj$ roots of $S_{-j}$. It has the same parity as $\sum c'_i-\sum c_i$. It follows from~\eqref{EqSumCiIsInteger} that $\sum c'_i-\sum c_i$ has the same parity as $2\sum c_i$. By our assumption $2\sum c_i$ is even, hence both $R_j$ and $S_{-j}$ have even number of roots.

Induction step is $j+1\to j$. We have $R_j(Z)=L_j(z)R_{j+1}(Z)$, $S_{-j}(Z)=\frac{S_{-j-1}(Z)}{M_{-j-1}(Z)}$. We should prove that $L_j$ and $M_{-j-1}$ have the same parity of number of roots. Recall that $\ovl{L_j}(Z^{-1})M_{-j-1}(Z)=P_{c'}(q^{-1-2j}Z)$. The Laurent polynomial $P_{c'}(Z)$ has even number of roots, this proves the induction step.

We turn to the second claim. We have \begin{multline*}
\frac{R_{j+2}(q^{-j-2}z)S_{-j-2}(q^{j+2}z)}{R_j(q^{-j}z)S_{-j}(q^j z)}=\\
\frac{R_{j+1}(q^{-j-2}z)S_{-j-1}(q^{j+2}z)}{R_j(q^{-j}z)S_{-j}(q^j z)}\frac{M_{-j-2}(q^{j+2}z)}{L_{j+1}(q^{-j-2}z)}=\\
\frac{B_{-j-1}(q^{j+2}z)M_{-j-2}(q^{j+2}z)}{A_j(q^{-j-2}z)L_{j+1}(q^{-j-2}z)}=\\
\frac{B_{-j-1}(q^{j+2}z)\ovl{A_j}(q^{-j-2}z^{-1})M_{-j-2}(q^{j+2}z)\ovl{L_{j+1}}(q^{-j-2}z^{-1})}{A_j(q^{-j-2}z)\ovl{A_j}(q^{-j-2}z^{-1})L_{j+1}(q^{-j-2}z)\ovl{L_{j+1}}(q^{-j-2}z^{-1})}
\end{multline*}
The denominator of this fraction is nonnegative and it follows from~\eqref{EqABequalsPQ} and~\eqref{EqLMEqualsPQ} that the numerator is $a^2P_{c'}(q^{1-j}z)P_c(q^{j-1}z)$. Since $P_{c'}(z)=\ovl{P_c}(z^{-1})$, the polynomial $P_{c'}(q^{1-j}z)P_c(q^{j-1}z)$ is nonnegative when $\abs{z}=1$.

It remains to prove the third claim. Similarly to the proof of Lemma~\ref{LemHowRSBehaves} we get that the roots of $R_j(q^{-j}z)S_{-j}(q^j z)$ are symmetric with respect to $z\mapsto \ovl{z}^{-1}$. Since $a^{-j}R_j(q^{-j}z)S_{-j}(q^j z)$ is balanced it is enough to check that its leading coefficient and its negative leading coefficient are complex conjugate. Hence it is enough to check that the leading coefficient and negative leading coefficient of $a^{-j}R_j(z)S_{-j}(z)$ have opposite arguments. We can prove this using induction. The base case $j=0$ is a choice of $\phi$. The induction step is
\begin{multline*}\frac{a^{-j-1}R_{j+1}(z)S_{-j-1}(z)}{a^{-j}R_j(z)S_{-j}(z)}=a^{-1}\frac{M_{-j-1}(z)}{L_j(z)}=\\
a^{-1}\frac{M_{-j-1}(z)\ovl{L_j}(z^{-1})}{L_j(z)\ovl{L_j}(z^{-1})}=\frac{P_{c'}(q^{-1-2j}z)}{L_j(z)\ovl{L_j}(z^{-1})}
\end{multline*}
Here the numerator's leading and negative leading coefficients are positive and the denominator is positive when $\abs{z}=1$, hence its leading and negative leading coefficients are complex conjugate.

\end{proof}



\subsection{Positivity condition for the Hermitian form}

The invariant Hermitian form $(\cdot,\cdot)$ is given by $(u,v)=T(u\phi(v))$, where $T$ is $g_t=\rho^2$-twisted. Here $t=ba^{-1}$. For fixed $t$ there are two conjugations $\rho$ with $\rho^2=g_t$. We denote one of them by $\rho_+$ and another by $\rho_-$.

It is enough to check that $(\cdot,\cdot)$ is positive definite on each space $Z^j \CN[Z,Z^{-1}]$. Let $m=x^j R_j(Z) R(Z)$. Then 
\begin{multline*}
(m,m)=T(x^j R_j(Z) R(Z)x^{-j}S_{-j}(Z)\ovl{R}(Z^{-1}))=\\
T(R_j(q^{-2j}Z)R(q^{-2j}Z)S_{-j}(Z)\ovl{R}(Z^{-1})).
\end{multline*}
Let $t=\abs{a} e^{2\pi i s}$. Similarly to the previous section, any trace  $T$ can be written as $T(R)=\int_C w(t)R(t)dt$, where $w$ is a certain quasi-periodic function. More precisely, $w(q^2x)=tw(x)$ and $w(qx)P(x)$ has no poles between $q^{-1}C$ and $qC$. In other words, $q^{2c_i}q^{2\ZN_{\geq 0}}$ is inside $C$ and $q^{2c_i}q^{2\ZN_{<0}}$ is outside of $C$.

We say that $i$ is bad if for a quasi-periodic function $w$ that has simple poles at $q^{2c_i+1+2\ZN}$ there exists $j$ such that $R_j(q^{-2j}x)S_{-j}(x)w(x)$ has poles between $C$ and $q^j S^1$, including the circle.

\begin{lem}
An index $i$ is bad if and only if $\Re c_i+\Re c'_i\le 0$ or $\Re c_i+\Re c'_i\geq 2$.
\end{lem}
\begin{proof}
The proof is similar to the proof of Lemma~\ref{LemDescriptionOfBadI}.
\end{proof}
\begin{prop}
\label{PropQBadPolesNoForm}
Suppose that there exists a bad $i$ such that $w$ has a pole at $q^{2c_i}$. Then $w$ does not give a positive definite form on $M_{c,c'}$.
\end{prop}
\begin{proof}
Suppose that $T$ defined by $w$ gives a positive definite form on $M_{c,c'}$. Let $j$ be a number that is bad for $i$.

Similarly to the proof of Proposition~\ref{PropBadPolesNoForm} we deduce that $w$ has no poles on $q^j S^1$ and write 
\begin{multline*}T\big(R_j(q^{-2j}Z)R(q^{-2j}Z)S_{-j}(Z)\ovl{R}(Z^{-1})\big)=\\
\int_{q^j S^1} R(q^{-2j}z)\ovl{R}(z^{-1})w(z)dz+\Phi\big(R(q^{-2j}z)\ovl{R}(z^{-1})\big),
\end{multline*} where $\Phi(S)$ is a nonzero  linear functional of the form $\Phi(S)=\sum c_i S(z_i)$. Here $z_i$ are the poles of $w$ between the contour $C$ and the circle $q^j S^1$. 

Let $L_0$ be a polynomial that has no zeroes on $S^1$ such that \[\Phi(\ovl{L_0(z^{-1})}\CN[z])=\Phi(L_0(q^{-2j}z)\CN[z])=\{0\}.\] We find a polynomial $L_1$ such that \(\Phi(L_1(q^{-2j}z)\ovl{L_1}(z^{-1}))\) does not belong to $\RN_{\geq 0}$. We use the Stone-Weierstrass theorem to find a polynomial $S$ such that $L_0S+L_1$ is uniformly small on $q^jS^1$. It follows that for $R=L_0S+L_1$ we get a contradiction with positivity of $T$.
\end{proof}
\begin{thr}
\label{ThrQAnswer}
Assume that there are $2k>0$ good indices. Fix any $k$ distinct pairs of numbers $(z_k,\ovl{z_k}^{-1})$. Then the cone of positive traces is isomorphic to the cone of elliptic functions with simple poles at $z_k,\ovl{z_k}^{-1}$ that are positive on $q^j S^1$ for all integers $j$. In particular, this cone has dimension $2k$, it does not depend on $s$ and on the particular choice of $P$, only on the number of good indices.

In the case when there are no good indices, there is a unique positive trace up to scaling in the case when $\rho^2=\id$ for one of the choices $\rho_+$, $\rho_-$, and no positive traces otherwise
\end{thr}
\begin{proof}
The proof is uniform for $2k>0$ and $2k=0$.

We know that a function $w$ with poles corresponding to bad indices does not give a positive definite form. Hence we can assume  that $w$ does not have poles at $q^{2c_i}$ for all bad $i$. 

Hence for any $j$ we have \begin{multline*}
T(R_j(q^{-2j}Z)R(q^{-2j}Z)S_{-j}(Z)\ovl{R}(Z^{-1}))=\\
\intl_{q^j S^1} R(q^{-2j}z)\ovl{R}(z^{-1})R_j(q^{-2j}z)S_{-j}(z)w(z)\abs{dz}.
\end{multline*} Using  the Stone-Weierstrass theorem we see that this quantity is positive for all nonzero Laurent polynomials $R$ if and only if $R_j(q^{-2j}z)S_{-j}(z)w(z)$ is nonnegative on $q^j S^1$. 

This is equivalent to saying that $R_j(q^{-j}z)S_{-j}(q^j z)w(q^j z)$ is nonnegative on $S^1$. We have \[\frac{w(q^{j+2}z)}{w(z)}=t=ba^{-1}=a^{-2}.\] Using the second statement of Lemma~\ref{LemQMIsomorphicToMRho} we get that \[\frac{R_j(q^{-j}z)S_{-j}(q^j z)w(q^j z)}{R_{j+2}(q^{-j-2}z)S_{-j-2}(q^{j+2}z)w(q^{j+2}z)}\] is nonnegative on $S^1$. Hence it is enough to check this condition for $j=0,1$.

Let $C$ be the cone of quasi-periodic functions that give a positive trace. We have $w\in C$ if and only if $R_0(z)S_0(z)w(z)$ and $R_1(q^{-1}z)S_{-1}(qz)w(qz)$ are nonnegative on $S^1$.  Using the third statement of Lemma~\ref{LemQMIsomorphicToMRho} we see that $P_0=R_0(z)S_0(z)$ and $P_1=e^{\pi i s}R_1(q^{-1}z)S_{-1}(qz)$ are real on $S^1$. 

Now we should describe the roots of $P_0$ and $P_1$ on the unit circle. Using the description of the roots of $R_j$ (and, changing $c$ and $c'$, roots of $S$) in Lemma~\ref{LemDescriptionOfRj} each root of $P_j$ corresponds to $c_i$ such that $2\Re c_i$ (or, equivalently, $\Re c'_i)$ is an integer with the same parity as $j$. Using reality condition~\eqref{EqRealityCondition} $c_i+\ovl{c'_\sigma(i)}=1$ we get that $c_i-c'_{\sigma(i)}$ is an integer. Since $c'$ is a generic parameter, we have $\sigma(i)=i$ and $c_i+\ovl{c'_i}=1$. Hence $\Re c_i+\Re c'_i=1$ and the index $i$ is good. 

In the case $c=c'$, the behavior of $R_0(z)S_0(z)w(z)$ and $R_1(q^{-1}z)S_{-1}(qz)w(qz)$ is described in Theorem~3.7 in~\cite{K}. The only change here is that the roots of $P_0,P_1$ on unit circle may be different, but both here and there the allowed poles of $w$ cancel out the roots of $P_0$, $P_1$. Hence the proof also works in out case and when $k>0$ we get that the cone of $w$ such that $P_0(z)w(z)$ and $e^{-\pi i s}P_1(z)w(qz)$ are nonnegative on $S^1$ has dimension $2k$ and does not depend on $s, P_0, P_1$. In particular, taking $s=0$, $P_0=P_1=1$ we get the cone in the statement of the theorem.

In the case $k=0$ the Theorem~3.7 in~\cite{K} contains a mistake that should be fixed as follows. Since there are no good indices, the only functions that could work are constant functions. Such trace exists only when $\rho^2$ is the identity. Since there are no good indices, there are no roots of $P_0$ and $P_1$ on the unit circle. Multiplying the isomorphism $\phi\colon M_{c,c',\rho^{-1}}\to M_{c',c}$ by a real number we can assume that $P_0$ is positive on the unit circle. For one of the choices of $\rho_+, \rho_-$, polynomial $P_1$ will also be positive on the unit circle and we get a positive trace.





\end{proof}
\paragraph{Example: $n=0,1$.}
Let's see how Theorems~\ref{ThrFilteredAnswer} and~\ref{ThrQAnswer} work for small values of $n$. We start with the case $n=1$ and discuss the case $n=2$ in the next subsection.

Note that for $q$-deformations $n$ should be even, so when $n=1$ we only have filtered deformations. In the case $n=1$  the algebra $A_c=W$ is isomorphic to Weyl algebra $W$, it is generated by $u,v$ with relation $[u,v]=1$. The parameter $c_1$ defines inclusion of $W$ into $\CN[x,x^{-1},\partial_x]\colon u\mapsto v,v\mapsto x^{-1}(x\partial_x-c_1)=\partial_x-c_1x^{-1}$. The bimodule $M_{c,c'}$ consists of differential operators with a possible pole at zero that send $x^{c_1'}\CN[x]$ to $x^{c_1}\CN[x]$. Since $x^{c_1-c_1'}$ provides a linear isomorphism between these two spaces, we have $M_{c,c'}=A_c x^{c_1-c_1'}=x^{c_1-c_1'} A_{c'}$, a free module of rank one from each side. Note that a linear isomorphism $A_c\cong M_{c,c'}\cong A_{c'}$ sends $a\in A_c$ to $x^{c_1'-c_1}ax^{c_1-c_1'}$, hence it is an isomorphism of algebras that sends $x$ to $x$ and $\partial_x-c_1x^{-1}$ to $\partial_x-c_1'x^{-1}$.

Note that the conjugation $\rho\colon A_c\to A_{c'}$ does not depend on $c,c'$: it sends $u$ to $bv$ and $v$ to $au$.

It follows that invariant sesquilinear forms on $M_{c,c'}$ are in one-to-one correspondence with invariant sesquilinear forms on $W$. In this case the only root of $P(x)=x$ is good and using Theorem~\ref{ThrFilteredAnswer} we deduce that positive traces exist only when $\rho=\rho_+$, $t\neq 1$ and in this case a positive trace is unique up to scaling. This is consistent with Proposition~4.7 in~\cite{EKRS}.

The case $n=0$ can also be considered, the algebra $A_c$ in this case is just the algebra ($q$-)differential operators with a possible pole at zero. In the case of differential operators there are no traces, in the case of $q$-differential operators there exists a trace only when $\rho^2$ is the identity. When $\rho^2=\id$, the trace is given by the formula $T(R(z))=[1]R=\int_{S^1}R(z)\frac{dz}{z}$ and is positive for one of the two possible choices of $\rho$.
\subsection{The case $n=2$, connection with unitary representations of $\SL(2)$ and $\SL_q(2)$}

Recall that in the case $n=2$ the ($q$-)deformations of the Kleinian singularity of type $A_1$ are central reductions of $U(\mf{sl}_2)$ and $U_q(\mf{sl}_2)$. In the case of deformations, Harish-Chandra $U(\mf{sl}_2)$-modules with integer weights of $\ad h$ are complex $(\mf{sl}_2(\C),\operatorname{SU}_2)$-modules, that is, Harish-Chandra modules in the classical sense. Unitarizable Harish-Chandra modules correspond to unitary representations of $\SL(2,\C)$. We check below that our results partially recover the classical results on the irreducible infinite-dimensional representations of $\SL(2,\C)$.

In the case of $q$-deformations, the situation is more complicated. The quantum $\SL_q(2)$ and $U_q(\mf{sl}_2)$ are dual Hopf algebras. Below we show that the classification of unitary representations in terms of the action of Casimir element is the same in our case of $U_q(\mf{sl}_2)$ and the case of $\SL_q(2)$ considered by Pusz~\cite{Pusz}. We leave the analytical details and precise relation between unitary representations of $U_q(\mf{sl}_2)$ and $\SL_q(2)$ to the future work. 

To avoid double counting of bimodules, we start with the following observation.

Note that for any half-integer $r$ we can change $h$ to $h+r$ in $A_c$ and change $h$ to $h-r$ in $A_{c'}$. In the case of $q$-deformations we change $Z$ to $q^{\pm r }Z$. The parameters $c_i$ will shift by $r$ and the parameters $c'_i$ will shift by $-r$. The conjugation $\rho$ is still defined, and we can still define $M_{c,c'}$ for these new parameters. We will denote this new bimodule by $M_{c,c',r}$ and the old one by $M_{c,c'}$. We claim that $\phi(m)=x^rmx^r$ is an isomorphism between $M_{c,c',r}$ and $M_{c,c'}$. It is enough to prove that it is a homomorphism. Indeed,

Then \[x\phi(m)=\phi(xm),\] \[h\phi(m)=hx^rmx^r=x^r(z+r)mx^r=\phi((z+r)m),\]
\[f\phi(m)=x^{-1}P(z-\tfrac12)x^rmx^r=x^{r-1}P(z+r-\tfrac12)mx^r=\phi(fm).\]

Hence we can shift parameters by a half-integer without changing anything.

Let us describe the parameters corresponding to bimodules $M_{c,c'}$ that admit an invariant positive definite form.

In the case $n=2$ there are parameters $c_1,c_2,c'_1,c'_2$ such that $c_1-c'_1$, $c_2-c'_2$ are integers, $c_1-c_2$ is not an integer. For some permutation $\sigma$ of $\{1,2\}$ we have $c_i+\ovl{c'_{\sigma(i)}}=1$. 

In the case when $\sigma$ is trivial we get $c_i+\ovl{c'_i}=1$ for $i=1,2$. Since $c_i-c'_i$ belongs to $\ZN$ we deduce that $2\Re c_i$ is an integer. Hence $\Re c_i$ and $\Re c'_i$ are integers or half-integers. An index $i$ is good if and only if $0<\Re c_i+\Re c'_i<2$. This is satisfied since $c_i+\ovl{c'_i}=1$.


Hence any $c_1$ and $c_2$ such that $2\Re c_1$, $2\Re c_2$ are integers, $c_1+c_2$ is an integer, $c_2-c_1$ is not an integer, give a unitarizable bimodule. This holds both for $q$-deformation and for a usual deformation. Note that we can shift $c_1$ and $c_2$ by $\frac{c_1+c_2-1}{2}$ and get the following: $c_1=c+\tfrac12$, $c_2=-c+\tfrac12$, $2\Re c$ is an integer, $c$ is not real. We expect that half-integer $c$ also gives a unitarizable bimodule. Note that $\Re c=0$ is a situation of regular bimodule $M_{c,c'}=A_c=A_{c'}$.

In the case when $\sigma$ is nontrivial we get $c_1+\ovl{c'_2}=1$, $c_2+\ovl{c'_1}=1$. 

From this we get that $\Im c_1=\Im c_2$. On the other hand, $c_2-c'_2$ is an integer, so $\Im c_2=\Im c'_2$. We assumed that $c_1+c_2$ is an integer, hence $c_1,c_1',c_2,c_2'$ are real numbers such that $c'_2=1-c_1$, $c'_1=1-c_2$. Shifting $c_1,c_2$ by $\frac{c_1+c_2-1}{2}$ we get $c_1=c+\tfrac12$, $c_2=-c+\tfrac12$, $c'_1=c+\tfrac12$, $c'_2=c-\tfrac12$. In this case the bimodule $M_{c,c'}$ is the regular bimodule $A_c=A_{c'}$.

An index $i$ is good if and only if $0<\Re c_i+\Re c'_i<2$. This is equivalent to $\abs{2c}<1$, so that $\abs{c}<\tfrac12$. This is the same answer we had in~\cite{EKRS} and~\cite{K}: the roots $\pm \alpha$ or $q^{\pm \alpha}$ of $P$ should satisfy $\abs{\Re\alpha}<\tfrac12$.

Let us compare these results with the classical results on irrreducible unitary representations of $\SL(2)$ and $\SL_q(2)$. The results for $\SL_q(2)$ are obtained in~\cite{Pusz}. The same article writes classification of unitary representations of $\SL(2)$ in terms of Casimir element, so we will also use it as a reference. Let $U_1(\mf{sl}_2)=U(\mf{sl}_2)$.

Similarly to Subsection 6.3 in~\cite{K} in the case of $q$-deformations we can take the locally finite part of $M_{c,c'}$ with respect to the adjoint action of the Hopf algebra $U_q(\mf{sl}_2)$ and $U_q$-invariant forms correspond to $g_{q^{-2}}$-twisted traces on $A_c$. Note that there are two traces on $A_c$, but a computation in~\cite{K} shows that the trace corresponding to the weight $w=x$ is zero on the locally finite part. Hence there one trace on the locally finite part of $A_c$ up to a constant, giving one positive definite form on the locally finite part of $M_{c,c'}$ up to a positive constant.

\begin{rem}
It is possible that one should consider instead an antilinear automorphism $\rho_S$ that multiplies $u$ and $v$ by a certain power of $Z$, as in~\cite{MeBulgaria}. In that paper we showed that the positive trace for $\rho_S$ is unique. We expect that the proof in this section can be combined with the methods of~\cite{MeBulgaria} to show that there is a unique $\rho_S$-invariant positive definite form in the case when all indices are good. For $n=2$ this means that the set of pairs $c,c'$ giving a unitarizable bimodule does not change.
\end{rem}

Now let us compute what finite-dimensional representations of $U_q(\mf{sl}_2)$ are inside $M_{c,c'}$ and what values of the Casimir element correspond to unitarizable parameters $c$.

There is freedom in choosing $h$, $Z$ in our case, but for $U_q(\mf{sl}_2)$ the element $h$ or $Z$ is fixed. In the case of $q=1$ element $h$ is uniquely defined by the condition that $ef+fe$ is an even polynomial in $h$. This means that $P(z+\tfrac12)+P(z-\tfrac12)$ is even, which is equivalent to $P$ being even. Since $P$ has roots $c_1-\tfrac12$, $c_2-\tfrac 12$, this means $c_1+c_2=1$. Our choice of $c_1,c_2$ satisfies $c_1+c_2=1$. In the case $q\neq 1$ the product of roots of $P$ should be one, this is also equivalent to $c_1+c_2=1$.

The $\ad h$ or $\operatorname{Ad} Z$ homogeneous elements in $M_{c,c'}$ have form $x^k R(z)$ or $x^k R(Z)$. They are highest weight if they commute with $u=x$, this happens if and only if $R=1$. Hence finite-dimensional representations of $U_q(\mf{sl}_2)$ are in one-to-one correspondence with nonnegative integers $k$ such that $x^k$ belongs to $M_{c,c'}$. This is equivalent to $c_i'+k\geq c_i$ for $i=1,2$. In the case of trivial $\sigma$ this means $k\geq \max(c_1-c'_1,c_2-c'_2)=\abs{2\Re c}$. In the case of nontrivial $\sigma$ the minimal $k$ is zero. We get that the minimal $k$ for the regular bimodule is zero. Conversely, when the minimal $k$ is zero, $M_{c,c'}$ is a regular bimodule. 


Suppose that the minimal $k$ is nonzero. Then $c=\frac{k}{2}+i\alpha$ for some real number $\alpha$. Let us compute the action of the Casimir element in this case. In the beginning of Section~6.1 of~\cite{K} we proved that the central reduction $U_q(\mf{sl}_2)/(\Omega-X_0)$ has parameter $P(z)=-\frac{z+z^{-1}-2}{(q-q^{-1})^2}+X_0$.

The polynomial $P$ has roots $q^{2\pm c}$. Using Vieta's formula we have $q^{2c}+q^{-2c}=2+X_0(q-q^{-1})^2$. Since $X_0$ depends linearly on $q^{2c}+q^{-2c}$, it is enough to describe the locus of $q^{2c}+q^{-2c}$. We have $q^{2c}+q^{-2c}=q^{k+2i\alpha}+q^{-k-2i\alpha}$. Let $q^k=r$, $q^{2i\alpha}=\cos\phi+i\sin\phi$. Then $q^{2c}+q^{-2c}=(r+r^{-1})\cos\phi+i(r-r^{-1})\sin\phi$. This is precisely the ellipse $\mc{E}_p$ in~\cite{Pusz}, equation~(0.3), multiplied by $\frac{\sqrt{1+q^2}}{q}$.

The set $\frac{\sqrt{1+q^2}}{q}\mc{E}_0$ is the closed interval $[-q-q^{-1},q+q^{-1}]$ and the endpoints correspond to one-dimensional representations. In the case of the regular bimodule the sum of roots in~\cite{K} belongs to $(-q-q^{-1},q+q^{-1})$. Hence our answer coincides with~\cite{Pusz} except that Pusz also allows positive half-integer $p$. If we add the square root of $Z$ we get exactly the same answer.

We already checked in~\cite{EKRS} that in the case of $q=1$ and the regular bimodule we get the classical theory of spherical unitary representation of $\SL(2,\CN)$. In the other cases we have $c=\frac{k}{2}+i\alpha$, where $k$ is a nonzero number. It is checked in~\cite{EKRS}, Example~2.1.2 that $P(x)=x^2-X_0$, where $X_0$ is the value of Casimir element. By Vieta's formula $X_0=c^2=\frac{k^2}{4}-\alpha^2+ik\alpha$. With $k=2p$ and $\alpha=t$ this coincides with $\frac{1}{2}\mc{P}_p-1$, where $\mc{P}_p$ are parabolas in~\cite{Pusz}.

{\footnotesize
\hspace{6mm}Department of Mathematics, Northwestern University, Evanston, IL 60208;

\texttt{klyuev@northwestern.edu}}
\end{document}